\documentclass[12pt]{amsart}

\usepackage{amsmath,amssymb,amsthm}
\usepackage[mathscr]{euscript}
\usepackage{color}
\usepackage{enumerate}
\usepackage{graphicx}
\usepackage{cancel}

\usepackage[margin=1in]{geometry}

\theoremstyle{definition}
\newtheorem{theorem}{Theorem}[section]
\newtheorem{prop}[theorem]{Proposition}

\newtheorem{conjecture}[theorem]{Conjecture}
\newtheorem{lemma}[theorem]{Lemma}
\newtheorem{definition}[theorem]{Definition}
\newtheorem{corollary}[theorem]{Corollary}
\newtheorem{remark}[theorem]{Remark}
\newtheorem{example}[theorem]{Example}
\newtheorem{question}[theorem]{Question}
\newtheorem{setup}[theorem]{Setup}

\numberwithin{equation}{section}

\newcommand{\conv}[1]{\mathrm{conv}\left\{#1\right\}}

\newcommand{\R}{\mathbb{R}}
\newcommand{\Z}{\mathbb{Z}}

\newcommand{\Q}{\mathbb{Q}}

\newcommand{\Pq}{\Delta_{(1,\bq)}}
\newcommand{\Pp}{\Delta_{(1,\bp)}}

\newcommand{\lcm}[1]{\mathrm{lcm}\left(#1\right)}
\newcommand{\res}[1]{\left\langle #1 \right\rangle}

\def\ds{\displaystyle}

\newcommand\commentout[1]{}

\newcommand{\bm}[1]{{\boldsymbol{#1}}}
\def\0{\bm{0}}
\def\1{\bm{1}}

\def\bc{\bm{c}}
\def\be{\bm{e}}
\def\bi{\bm{i}}
\def\bj{\bm{j}}

\def\bp{\bm{p}}
\def\bq{\bm{q}}
\def\br{\bm{r}}
\def\bs{\bm{s}}

\def\bx{\bm{x}}
\def\by{\bm{y}}

\def\brho{\bm{\rho}}

\begin{document}



\title{$h^*$-Polynomials With Roots on the Unit Circle}

\author{Benjamin Braun}
\address{Department of Mathematics\\
  University of Kentucky\\
  Lexington, KY 40506--0027}
\email{benjamin.braun@uky.edu}

\author{Fu Liu}
\address{Department of Mathematics, University of California, One Shields Avenue, Davis, CA 95616}
\email{fuliu@math.ucdavis.edu}



\date{19 July 2018}

\thanks{
  The first author was partially supported by grant H98230-16-1-0045 from the U.S. National Security Agency.
  The second author was partially supported by a grant from the Simons Foundation \#426756.
  This material is also based in part upon work supported by the National Science Foundation under Grant No. DMS-1440140 while both authors were in residence at the Mathematical Sciences Research Institute in Berkeley, California, during the Fall 2017 semester. 
}

\begin{abstract}
For an $n$-dimensional lattice simplex $\Pq$ with vertices given by the standard basis vectors and $-\bq$ where $\bq$ has positive entries, we investigate when the Ehrhart $h^*$-polynomial for $\Pq$ factors as a product of geometric series in powers of $z$.
Our motivation is a theorem of Rodriguez-Villegas implying that when the $h^*$-polynomial of a lattice polytope $P$ has all roots on the unit circle, then the Ehrhart polynomial of $P$ has positive coefficients.
We focus on those $\Pq$ for which $\bq$ has only two or three distinct entries, providing both theoretical results and conjectures/questions motivated by experimental evidence.
\end{abstract}

\maketitle

\tableofcontents


\section{Introduction}

\subsection{Background and Motivation}
Assume for this paper that $P$ is a full-dimensional lattice polytope in $\R^n$, i.e. $P$ is given by the convex hull of a finite subset of $\Z^n$ and the affine hull of $P$ has dimension $n$.
Letting $tP$ denote the dilation of $P$ by $t$, the \emph{Ehrhart polynomial} $L_P(t)$ is defined to be the degree $n$ polynomial satisfying
\[
  L_P(t):=|tP\cap \Z^n|
\]
for $t\in \Z_{\geq 1}$, which is known to exist due to work of Ehrhart~\cite{Ehrhart}.
Much is known about the roots and coefficients of Ehrhart polynomials, but major open questions remain.
One area of active investigation~\cite{LiuPositivitySurvey} is to identify criteria that imply $L_P(t)\in \Q_{> 0}[t]$, in which case we say that $P$ is \emph{Ehrhart positive}.

Given a polynomial $f(t)\in \R [t]$ of degree $n$, if all the roots of $f(t)$ have negative real parts, then expanding $f(t)$ as a product of terms of the form $(t+r)$ and $(t+r+bi)(t+r-bi)$ implies that $f(t)\in \R_{> 0}[t]$.
Thus, Ehrhart positivity is a consequence when $L_P(t)$ has roots with only negative real parts.
One approach to investigating those $P$ such that $L_P(t)$ has roots with only non-negative real parts is to consider the generating function for $L_P(t)$.
For any polynomial $f(t)\in \R [t]$ of degree $n$, there exist values $h^*_j\in \R$ with $\sum_{j=0}^nh_j^*\neq 0$ such that
\[
  \sum_{t=0}^\infty f(t)z^t=\frac{\sum_{j=0}^nh_j^*z^j}{(1-z)^{n+1}} \, .
\]
When $f(t)=L_P(t)$, it is known due to work of Stanley~\cite{StanleyDecompositions} that $h_j^*\in \Z_{\geq 0}$ for all $j$, and we refer to the polynomial $h^*(P;z):=\sum_{j=0}^nh_j^*z^j$ as the \emph{$h^*$-polynomial of $P$}.
Further, $h_0^*=1$ and $h_n^*=|\mathrm{int}(P)\cap \Z^n|$ where $\mathrm{int}(P)$ denotes the topological interior of $P$.
Our connection to Ehrhart positivity is provided by the following theorem, which is a special case of a more general result proved by Rodriguez-Villegas.

\begin{theorem}[Rodriguez-Villegas \cite{Rodriguez-Villegas}]
  \label{thm:geomriemhyp}
  If $f(t)\in \R[t]$ is of degree $n$ and the associated polynomial $\sum_{j=0}^nh_j^*z^j$ is also of degree $n$ with all roots on the unit circle, then the roots of $f(t)$ all have real part equal to $-1/2$.
\end{theorem}

As a consequence of Ehrhart-MacDonald Reciprocity, those lattice polytopes $P$ whose Ehrhart polynomials have roots with real parts equal to $-1/2$ form a subfamily of the class of reflexive polytopes, where $P$ is \emph{reflexive} if some translate $P'$ of $P$ by an integer vector contains the origin in its interior and satisfies that the polar dual of $P'$ is also a lattice polytope.
By a result due to Hibi~\cite{HibiDualPolytopes}, it is known that $P$ is reflexive if and only if $h_i^*=h_{n-i}^*$ for all $i$.
Since $h^*_0=1$ for all lattice polytopes, it follows that reflexive $P$ have $h_n^*=1$.
Lattice polytopes satisfying $h_n^*=|\mathrm{int}(P)\cap \Z^n|=1$ are called \emph{canonical Fano} polytopes, and thus reflexive polytopes are contained within this broader class.

To summarize, if one can apply Theorem~\ref{thm:geomriemhyp} to $L_P(t$), then we must have that $h^*(P;z)$ is monic of degree $n$ with all of its roots on the unit circle.
The $h^*$-polynomials with these properties fall within a large and well-studied family.

\begin{definition}
  A \emph{Kronecker polynomial} is a monic integer polynomial with all roots inside the complex unit disk.
\end{definition}

It is known as a consequence of results due to Hensley~\cite{HensleyInterior} and Lagarias and Ziegler~\cite{lagariasziegler} that for each dimension $n$, there are only a finite number of canonical Fano polytopes (up to unimodular equivalence).
The following classical theorem complements this fact.

\begin{theorem}[Kronecker \cite{KroneckerRoots}, Damianou \cite{DamianouUnitCircle}]
  \label{thm:Kronecker}
  For each fixed $n$, there are only finitely many Kronecker polynomials of degree $n$.
  Further, if $h(z)\in \Z[z]$ is a Kronecker polynomial, then all the roots of $h(z)$ are roots of unity, and $h(z)$ factors as a product of cyclotomic polynomials.
\end{theorem}

Combining Theorem~\ref{thm:geomriemhyp} and Theorem~\ref{thm:Kronecker} in the setting of Ehrhart $h^*$-polynomials, we obtain the following corollary.

\begin{corollary}[see Corollary 2.2.4 in \cite{LiuPositivitySurvey}] \label{cor:reflexive}
  If the $h^*$-polynomial of a canonical Fano polytope is a Kronecker polynomial, then $P$ is reflexive and $L_P(t)$ is Ehrhart positive.
\end{corollary}


\subsection{Our Contributions}

One way for an $h^*$-polynomial to be Kronecker is to factor as a product of geometric series in powers of $z$, which we refer to as a \emph{geometric factorization}.
Motivated by Corollary~\ref{cor:reflexive}, we explore geometric factorizations for lattice simplices of the following form: let $\Pq$ be the simplex with vertices given by the standard basis vectors and $-\bq$ where $\bq$ has positive entries.
These simplices are related to fans defining weighted projective spaces, and their Ehrhart-theoretic properties have recently been studied by Payne~\cite{Payne}, Braun, Davis, and Solus~\cite{BraunDavisSolusIDP}, Solus~\cite{SolusNumeralSystems}, and Balletti, Hibi, Meyer, and Tsuchiya~\cite{laplaciandigraphs}.

In Section~\ref{sec:pq}, we establish basic facts about the $h^*$-polynomials of these simplices and review some of their properties related to $h^*(\Pq;z)$ being Kronecker.
In Section~\ref{sec:factoring}, we prove that when $\Pq$ is reflexive there is always a geometric series that can be factored from  $h^*(\Pq;z)$, leading us to define a polynomial $g_{\br}^{\bx}(z)$ that is our primary object of study.

Sections~\ref{sec:twointegers} and~\ref{sec:twoodd} contain our main theoretical results, focused on $\bq$-vectors with two distinct entries $a$ and $ka-1$.
In Section~\ref{sec:twointegers}, we identify four families of $\bq$-vectors for which $h^*(\Pq;z)$ factors as a product of geometric series.
In Section~\ref{sec:twoodd}, we prove that when $\bq$ has distinct entries $2$ and $2k-1$, these families essentially classify those simplices with Kronecker $h^*$-polynomials.

In Section~\ref{sec:questions}, we provide various conjectures and questions informed by experiments using SageMath~\cite{sagemath}.
These include conjectured extensions of our result in Section~\ref{sec:twoodd}, a conjectured Kronecker family related to Fibonacci numbers, and an exploration of the case where $\bq$ has three distinct entries, among other topics.


\section{The Simplices $\Pq$}\label{sec:pq}

\subsection{Definition and Reflexivity}
Given a vector of positive integers $\bq\in \Z_{>0}^n$, we define
\[
  \Pq := \conv{\be_1,\ldots,\be_n,-\sum_{i=1}^n q_i\be_i}
\]
where $\be_i$ denotes the $i$-th standard basis vector in $\R^n$. 
There is a natural stratification of the family of simplices of the form $\Pq$ based on the distinct entries in the vector $\bq$.
Given a vector of distinct positive integers $\br = (r_1,\ldots,r_d)$, write
\[
  (r_1^{x_1},r_2^{x_2},\ldots,r_d^{x_d}):=(\underbrace{r_1,r_1,\ldots,r_1}_{x_1\text{ times}},\underbrace{r_2,r_2,\ldots,r_2}_{x_2\text{ times}},\ldots,\underbrace{r_d,r_d,\ldots,r_d}_{x_d\text{ times}}) \, .
\]

\begin{definition}
  We say that both $\bq$ and $\Pq$ are \emph{supported} by the vector $\br = (r_1,\ldots,r_d)$ with \emph{multiplicity} $\bx=(x_1, \dots, x_d)$ if $\bq=(q_1,\ldots,q_n)=(r_1^{x_1},r_2^{x_2},\ldots,r_d^{x_d})$.
\end{definition}

Since our goal is to determine when $h^*(\Pq;z)$ is a Kronecker polynomial, Corollary~\ref{cor:reflexive} implies that we are only interested in the case where $\Delta_{(1,\bq)}$ is reflexive. 
It is straightforward to show~\cite{conrads} that $\Delta_{(1, \bq)}$ is reflexive if and only if
\begin{equation}
  q_i \text{ divides } 1 + \sum_{j =1}^n q_j, \quad \text{ for all $1 \le i \le n$ } \, .
  \label{equ:conrads}
\end{equation}
Equivalently, if $\bq$ is supported by $\br$ with multiplicity $\bx$, then $\Pq$ is reflexive if and only if if $\lcm{r_1,\dots, r_d}$ divides $1+\sum_{i=1}^d x_i r_i$, which leads us to the following definition.

\begin{definition}
  Say $\bx$ is an \emph{R-multiplicity} of $\br$ if $\lcm{r_1,\dots, r_d}$ divides $1+\sum_{i=1}^d x_i r_i$.
\end{definition} 

Throughout the rest of this paper, we will frequently use the following setup.
\begin{setup}
  \label{setup1}
  Let $\bq$ be supported by the vector $\br=(r_1, \dots, r_d) \in \left( \Z_{>0} \right)^d$ with an R-multiplicity $\bx=(x_1, \dots, x_d) \in  \left( \Z_{>0} \right)^d$.
  Let $\ell = \ell(\bq)$ be the integer defined by
  \begin{equation}
    1+\sum_{i=1}^d x_ir_i=\ell \cdot \lcm{r_1,r_2,\dots,r_d} \, .
    \label{equ:elldefn}
  \end{equation}
  Finally, we define 
  \begin{equation}
    \bs := (s_1, \dots, s_d), \quad \text{where $s_i := \lcm{r_1, \dots, r_d}/ r_i$ for each $1 \le i \le d$}.
    \label{equ:sdefn}
  \end{equation}
\end{setup}

\begin{lemma}
  \label{lem:gcdlcm}
  Using Setup~\ref{setup1}, we have that $\gcd(r_1, \dots, r_d) = 1$ and thus
  \begin{equation}
    \lcm{s_1, \dots, s_d}= \lcm{r_1, \dots, r_d}.
    \label{equ:lcmeq}
  \end{equation}
\end{lemma}

\begin{proof}
  It follows from \eqref{equ:elldefn} that $\gcd(r_1,\dots,r_d)$ has to be $1.$ By the definition of $s_i,$ we can verify that
  \[ \lcm{s_1,\dots,s_d} = \frac{\lcm{r_1,\dots,r_d}}{\gcd(r_1,\dots,r_d)} = \lcm{r_1,\dots,r_d}. \qedhere \]
\end{proof}

Our analysis of families of $\bq$-vectors will require a precise language for studying R-multiplicities of vectors, which we introduce next.

\begin{definition}
  We define $\res{n}:=\{0,1,\dots, n-1\}$ and $[-n]:=\{-n, -(n-1), \dots, -1\}$.
\end{definition}

\begin{definition}\label{defn:division}
  Suppose $\bs = (s_1, \dots, s_d)$ is a vector of positive integers and $\bx = (x_1, \dots, x_d)$ is a vector of integers. 
  Let $\bc = (c_1, \dots, c_d)$ and $\brho = (\rho_1, \dots, \rho_d)$ be two vectors of integers such that for each $i,$
  \begin{equation}
    x_i = c_i s_i + \rho_i.  
    \label{equ:divx}
  \end{equation}
  We say such a pair $(\bc, \brho)$ is an \emph{$\bs$-division of $\bx$}, and $\brho$ is an \emph{$\bs$-remainder} and $\bc$ is an \emph{$\bs$-quotient}. 
  It is clear that any valid $\bs$-quotient or $\bs$-remainder determines a unique $\bs$-division. 
  However, $\bs$-divisions exist nonuniquely. 

  Suppose further $\br = (r_1, \dots, r_d)$ is a vector of positive integers such that $\br$ and $\bs$ are related as in \eqref{equ:sdefn}. 
  We say $\brho$ (or $\bc$ or $(\bc,\brho)$) is \emph{desirable} if 
  \[ \sum_{i=1}^d \rho_i r_i = -1.\]
\end{definition}

\begin{example}
  \label{ex:aka-1desire}
  Assume Setup~\ref{setup1} with $\br=(a, ka-1)$ for some positive integers $a$ and $k$. 
  Then $r_1 = s_2 = a$ and $r_2 = s_1 = ka-1$.
  Suppose $\bx=(c_1(ka-1)-k,c_2a+1)$. (In fact, one can show that any R-multiplicity of $\bx$ is in the form. See Subsection \ref{subsec:setup} and Example \ref{ex:uniquerho}.) 
  Then there is a desirable $\bs$-division of $\bx$ with
  \[ 
    \rho_1 = -k, \quad \rho_2 = 1 \, ,
  \]
  which follows from observing that $(-k) a + 1 \cdot (ka-1) = -1.$
\end{example}

\begin{lemma}\label{lem:desrem}
  Suppose two vectors of positive integers $\bs = (s_1,\dots,s_d)$ and $\br = (r_1, \dots, r_d)$ are related as in \eqref{equ:sdefn}. 
  Then a vector of integers $\brho$ satisfies
  \begin{equation}\label{equ:congto1} 
    \sum_{i=1}^d \rho_i r_i \equiv -1 \quad \mod \lcm{r_1, \dots, r_d} 
  \end{equation}
  if and only if $\brho$ is an $\bs$-remainder of some R-multiplicity $\bx$ of $\br.$
  Moreover, if $\bx$ is an R-multiplicity of $\br,$ there exists a desirable $\bs$-remainder $\brho$ of $\bx$ such that for each $i,$
  \[ \rho_i \in \res{s_i} \text{ or } [-s_i].\]
\end{lemma}

\begin{proof}
  Suppose $\brho$ is an $\bs$-remainder of some R-multiplicity $\bx$ of $\br.$ Plugging in $x_i = c_i s_i + \rho_i$ and using the fact that $s_i r_i = \lcm{r_1, \dots, r_m}$, we obtain
  \[
    1 + \sum_{i=1}^dx_i r_i = 1+ \sum_{i=1}^d (c_i s_i r_i + \rho_i r_i) \equiv 1+ \sum_{i=1}^d \rho_i r_i \mod \lcm{r_1, \dots, r_d} \, .
  \]
  Thus, \eqref{equ:congto1} follows from the fact that $\bx$ is an R-multiplicity. Conversely, if \eqref{equ:congto1} holds, one sees that $\brho$ is an $\bs$-remainder of $\bx = \brho$ which is an R-multiplicity of $\br.$

  We next show the existence of our specified desirable remainder.
  Let $(\bc, \brho)$ be the (unique) $\bs$-division of $\bx$ such that $\rho_i \in \res{s_i}$ for each $i$.
  As $0 \le \rho_i < s_i,$ we have $0 \le \rho_i r_i < s_i r_i = \lcm{r_1, \dots, r_d}$. 
  Hence,
  \[ 
    0 \le \sum_{i=1}^d \rho_i r_i \le d \cdot \lcm{r_1, \dots, r_d} - d\, .
  \]
  Thus, Equation~\eqref{equ:congto1} implies that $\sum_{i=1}^d \rho_i r_i = m \cdot \lcm{r_1, \dots, r_d} - 1$ for some $1 \le m \le \max(1,d-1)$.
  Note that for each $i$,
  \[ 
    x_i = c_i s_i + \rho_i = (c_i+1) s_i + (\rho_i - s_i) \, ,
  \]
  where $\rho_i - s_i \in [-s_i]$. 
  It is straightforward to verify that if we let $(\bc', \brho')$ be the $\bs$-division of $\bx$ obtained from $(\bc, \brho)$ by choosing $m$ indices $j_1, \dots, j_m$ and replacing each $(c_{j_p}, \rho_{j_p})$ with $(c_{j_p}+1, \rho_{j_p} - s_{j_p})$, then $(\bc', \brho')$ is desirable and satisfies that $\rho_i' \in \res{s_i} or [-s_i]$ for each $i$.
%
\end{proof}

\begin{lemma}\label{lem:desell}
  Assume Setup~\ref{setup1}.
  Suppose $(\bc, \brho)$ is a desirable $\bs$-division of $\bx.$ Then
  \[
    \ell = \ell(\bq) = \sum_{i=1}^d c_i \, .
  \]
\end{lemma}
\begin{proof}
  $\displaystyle 1 + \sum_{i=1} x_i r_i =  1 + \sum_{i=1} (c_i s_i + \rho_i) r_i = \left( \sum_{i=1} c_i  \right)\cdot \lcm{r_1,\dots, r_d} + \cancelto{0}{\left( 1 + \sum_{i=1}^d \rho_i r_i \right)}.$ 
\end{proof}

\begin{example}
  Building on Example~\ref{ex:aka-1desire} where $\br=(a,ka-1)$ and $\bx=(c_1(ka-1)-k,c_2a+1)$, it is elementary to verify that
  \[
    1+(c_1(ka-1)-k)a+(c_2a+1)(ka-1)=(c_1+c_2)a(ka-1) \, .
  \]
\end{example}

\subsection{$h^*$-Polynomials and Geometric Factorizations}

The following theorem shows that the $h^*$-polynomial for any $\Pq$ can be expressed purely in terms of the vector $\bq$.

\begin{theorem}[Braun, Davis, and Solus \cite{BraunDavisSolusIDP}] \label{thm:qhstar}
  The $h^*$-polynomial of $\Pq$ is given by 
  \[
    \sum_{b=0}^{q_1+q_2+\cdots +q_n}z^{w(b)}
  \]
  where
  \begin{equation}\label{eqn:bweight}
    w(b)=b-\sum_{i=1}^n\left\lfloor\frac{bq_i}{1+\sum_{j=1}^n q_j} \right\rfloor \, .
  \end{equation}
\end{theorem}

\begin{example}
  \label{ex:payne}
  For integers $w\geq0$, $a\geq 3$, and $t\geq w+2$, Payne~\cite{Payne} introduced the reflexive simplex $\Pq$ with 
  \begin{equation}
    \label{eqn: payne}
    \bq=(\underbrace{1,1,\ldots,1}_{at-1 \text{ times}},\underbrace{a,a,\ldots,a}_{w+1 \text{ times}}) \, ,
  \end{equation}
  in other words we have $\br=(1,a)$ with R-multiplicity $\bx=(at-1,w+1)$.
  It follows from Theorems~\ref{thm:ellgeomfactor} and \ref{thm:case0} below that
  \[
    h^*(\Pq;z)=(1+z^t+z^{2t}+\cdots +z^{(a-1)t})(1+z+z^2+\cdots +z^{t+w})\, .
  \]
\end{example}

In this work we are primarily interested in studying when $h^*(\Pp;z)$ factors as a product of geometric series, similarly to Payne's simplices in Example~\ref{ex:payne}.
We next define language and notation for working with products of geometric series in varying powers of $z$.

\begin{definition} For any $e \in \Z_{>0}$ and $\gamma \in \Z_{\ge 2},$ we call
  \[ \sum_{i=0}^{\gamma-1} z^{i e} = 1 + z^e + z^{2e} + \cdots + z^{(\gamma-1) e} \]
  a \emph{geometric series (in powers of $z$)} of \emph{length} $\gamma$ and with \emph{exponent} $e.$	
  We say a polynomial $f(z)$ in $z$ is a \emph{product of geometric series (in powers of $z$)} if there exists $p \in \Z_{>0}$, $e_1, e_2, \dots, e_p \in \Z_{>0}$ and $\gamma_1, \dots, \gamma_p  \in \Z_{\ge 2}$ such that
  \begin{equation}
    f(z) = \prod_{j=1}^p \sum_{i=0}^{\gamma_j-1} z^{i e_j} = \prod_{j=1}^p \left( 1 + z^{e_j} + z^{2 e_j} + \cdots z^{(\gamma_j-1) e_j} \right).
    \label{equ:geofac}
  \end{equation}
  We also call the right hand side of the above equation a \emph{geometric factorization} of $f(z)$. 
\end{definition}

We remark that geometric factorizations of a polynomial $f$ are not necessarily unique, e.g, $f(z) = 1 + z + z^2 + z^3$ is a geometric series itself, but can also be expressed as $(1+z)(1+z^2)$.
As our first observation regarding geometric factorizations, we show that ordinary geometric series are $h^*$-polynomials for only one family of $\Pq$ simplices.

\begin{prop}\label{prop:oneinteger}
  Assume Setup \ref{setup1}. Then $h^*(\Pq; z)$ is a geometric series if and only if $\bq$ is supported on one integer.
\end{prop}

\begin{proof}
  Suppose $\bq$ is supported on one integer $r,$ i.e. $q=(r^x)$ for some positive integers $r$ and $x$.
  Since $r$ divides $1 + xr,$ we have that $r = 1$ and $x$ can be any positive integer. Applying Theorem \ref{thm:qhstar}, we immediately obtain that 
  \[ h^*(\Pq; z) = \sum_{b=0}^{x} z^{w(b)} =  \sum_{b=0}^{xr} z^{b},\]
  which is a geometric series of length $1+xr$ and with exponent $1$.

  Conversely, assume $h^*(\Pq; z)$ is a geometric series. Note that 
  \[ w(1)=1-\sum_{i=1}^d x_i \left\lfloor\frac{r_i}{1+\sum_{j=1}^d x_j r_j} \right\rfloor \, = 1. \]
  Hence, $z^1$ appears in $h^*(\Pq; z)$. 
  This implies that $h^*(\Pq; z)$ is a geometric series with exponent $1.$ 
  Thus, we must have that for each $b$ with $0 \le b \le \sum_{j=1}^d x_j r_j,$
  \[ w(b)=b-\sum_{i=1}^d x_i \left\lfloor\frac{br_i}{1+\sum_{j=1}^d x_j r_j} \right\rfloor \, = b.\]
  Thus, $b r_i < 1 + \sum_{j=1}^d x_j r_j$ for all such $b$.
  Considering the case where $b = \sum_{j=1}^d x_j r_j$, we must have $r_i=1$ for all $i$.
  Hence $\Pq$ is supported on one integer $r=1.$
\end{proof}

\subsection{Free Sums Create New Kronecker $h^*(\Pq;z)$}
For two reflexive simplices $\Pq$ and $\Pp$ with Kronecker $h^*$-polynomials, there exists an operation that produces a new simplex $\Delta_{(1,\by)}$ that is reflexive with a Kronecker $h^*$-polynomial.
We say that $P\oplus Q:=\conv{P\cup Q}$ is an \emph{affine free sum} if, up to unimodular equivalence, $P\cap Q = \{0\}$ and the affine span of $P$ and $Q$ are orthogonal coordinate subspaces of $\R^n$.
Suppose further that $P\subset \R^n$ and $Q\subset\R^m$ are reflexive polytopes with $0\in P$ and the vertices of $Q$ labeled as $v_0,v_1,\ldots,v_m$.  
For every $i=0,1,\ldots,m$, we define the polytope
\[
  P\ast_i Q:=\conv{(P \times 0^m)\cup(0^n\times Q-v_i)}\subset\R^{n+m}.
\]
The following theorem indicates that affine free sum decompositions can be detected from the $\bq$-vector defining $\Pq$ and induce a product structure for $h^*$-polynomials.
\begin{theorem}[Braun, Davis \cite{BraunDavisReflexive}]
  \label{thm:freesumdecomp}
  If $\Pp$ and $\Pq$ are full-dimensional reflexive simplices with $\bp=(p_1,\ldots,p_n)$ and $\bq=(q_1,\ldots,q_m)$, respectively, then $\Pp *_0 \Pq$ is a reflexive simplex $\Delta_{(1,\by)}$ with $\by=(p_1,\ldots,p_n,sq_1,\ldots,sq_m)$ where $s = 1+ \sum_{j=1}^n p_j$.
  Moreover, if $\Delta_{(1,\by)}$ arises in this form, then it decomposes as a free sum.
  Further, if $\Pp$ and $\Pq$ are reflexive, then $h^*(\Pp *_0 \Pq;z)=h^*(\Pp;z)h^*(\Pq;z)$.
\end{theorem}

\begin{corollary}
  If $h^*(\Pp;z)$ and $h^*(\Pq;z)$ are Kronecker polynomials, then we also have that $h^*(\Pp *_0 \Pq;z)$ is a Kronecker polynomial.
\end{corollary}

\begin{remark}
  More generally, if $P$ and $Q$ are reflexive polytopes, then free sums of $P$ and $Q$ have $h^*$-polynomials obtained as products of the $h^*$-polynomials of their free summands.
  Thus, the resulting $h^*$-polynomials are also Kronecker when the summands have Kronecker $h^*$-polynomials.
\end{remark}


\section{Factoring $h^*(\Pq;z)$ for Reflexive $\Pq$}\label{sec:factoring}

\subsection{Reflexive $\Pq$  Always Have a Geometric Series Factor in $h^*(\Pq;z)$}
In this subsection, we show that for a reflexive $\Pq$, it is always possible to factor a geometric series from $h^*(\Pq;z)$.
The following polynomial plays a fundamental role in this factorization.
\begin{definition}\label{def:ellpoly}
  Suppose $\br, \bx$, $\ell$ and $\bs$ are as given in Setup \ref{setup1}.
  We define 
  \[
    g_\br^\bx(z):=\sum_{0\leq \alpha < \lcm{r_1,\ldots,r_d}}z^{u(\alpha)}
  \]
  where
  \[
    u(\alpha) = u_\br^\bx(\alpha):=\alpha \ell - \sum_{i=1}^d x_i\left\lfloor  \frac{\alpha}{s_i} \right\rfloor \, .
  \]
\end{definition}

\begin{theorem}\label{thm:ellgeomfactor}
  Assuming Setup \ref{setup1}, we have that
  \[
    h^*(\Pq;z) = \left(\sum_{t=0}^{\ell-1}z^t\right) \cdot g_\br^\bx(z).
  \]
\end{theorem}

\begin{proof} 
  Let $M := \lcm{r_1,\dots,r_d}.$
  Let $0\leq b< \ell M$ and write $b=\alpha \ell + \beta$ for $0\leq \alpha<M$ and $0\leq \beta <\ell$.
  Then using~\eqref{eqn:bweight} we have:
  \begin{align*}
    w(b)=w(\alpha\ell+\beta) & =  \alpha\ell+\beta - \sum_{i=1}^d x_i\left\lfloor  \frac{(\alpha\ell+\beta)r_i}{\ell M} \right\rfloor \\
                             & =  \beta + \alpha\ell - \sum_{i=1}^d x_i\left\lfloor  \frac{\alpha\ell+\beta}{\ell s_i} \right\rfloor 
                               =  \beta + \alpha\ell - \sum_{i=1}^d x_i\left\lfloor  \frac{\alpha}{s_i} + \frac{\beta}{\ell} \frac{1}{s_i}\right\rfloor. 
  \end{align*}
  Since $0\leq \beta<\ell$, we have
  \[
    0\leq \frac{\alpha}{s_i} + \frac{\beta}{\ell} \frac{1}{s_i} < \frac{\alpha+1}{s_i}
  \]
  and thus 
  \[
    w(b)=w(\alpha\ell+\beta)=\beta + \alpha\ell - \sum_{i=1}^d x_i\left\lfloor  \frac{\alpha}{s_i}\right\rfloor = \beta + u(\alpha) \, .
  \]
  Hence, it follows from Theorem~\ref{thm:qhstar} that
  \[
    h^*(\Pq;z) = \sum_{b=0}^{\ell s}z^{w(b)}  = \sum_{\substack{0\leq \alpha < M \\ 0\leq \beta < \ell}}z^{\beta + u(\alpha)} = \left( \sum_{ 0\leq \beta < \ell}z^{\beta} \right) g_\br^\bx(z). \qedhere
  \]
\end{proof}

The following is an immediate consequence of Theorem~\ref{thm:ellgeomfactor}.

\begin{corollary}
  For $\bq=(r_1^{x_1},\ldots,r_d^{x_d})$, we have $h^*(\Pq;z)$ is a Kronecker polynomial if and only if $g_\br^\bx(z)$ is a Kronecker polynomial.
\end{corollary}

\begin{remark}\label{remark:h*}
  If $h^*(\Pq;z)$ has a geometric factorization, then $g_\br^\bx(z)$ does not necessarily have a geometric factorization, although the converse is clearly true. The smallest counterexample is when $\br = (2,5)$ and $\bx= (7,5).$ In this case,
  \begin{align*}
    h^*(\Pq; z) =& 1 + z + 2 z^2 + 4z^3 + 4z^4 + 5z^5 + 6 z^6 + 5z^7 + 4z^8 + 4 z^9 + 2z^{10} + z^{11} + z^{12},
                   \intertext{which can be factored as $(1 + z^2) (1+z^3)^2 (1+z+ z^2 +z^3+z^4)$, and}
                   g_\br^\bx(z) =& 1 + z^2 + 2z^3 + z^4 + z^5 + 2z^6 + z^7 + z^9, 
  \end{align*}
  which cannot be written as a product of geometric series. 
\end{remark}

\begin{remark}
  Another area of interest is identifying lattice polytopes where $h^*(P;z)$ has only real roots; see recent work by Solus~\cite{SolusNumeralSystems} for an investigation of $\Pq$ with this property.
  Theorem~\ref{thm:ellgeomfactor} implies that if $\Pq$ is reflexive with $\ell\geq 3$, then $h^*(\Pq;z)$ is not real-rooted.
Further, while our primary focus in this paper is on factoring $h^*$-polynomials as products of geometric series, there are techniques related to real-rootedness that count the number of unit circle roots of a given polynomial.
For example, if $f(z)$ is degree $n$ and does not have $1$ as a root, then the transformation $\displaystyle g(z)=(z+i)^nf\left(\frac{z-i}{z+i}\right)$ sends unit circle roots of $f$ to real roots of $g$~\cite[Page 7]{ConradsUnitRootsNotes}.
Thus, in this setting $f$ has all unit circle roots if and only if $g$ has only real roots.
It would be of interest to determine if these techniques can be applied productively in the setting of $h^*$-polynomials.
\end{remark}

The next result shows that extending $\bq=(\br,\bx)$ by $\lcm{r_1,\ldots,r_d}$ does not alter the structure of $g_\br^\bx(z)$.

\begin{theorem}\label{thm:lcmextend}
  Let $\bq=(r_1^{x_1},\ldots,r_d^{x_d})$ where $\bx$ is an R-multiplicity of $\br$ and $\ell=\ell(\bq)$.
  Then $\bq '=(r_1^{x_1},\ldots,r_d^{x_d},\lcm{r_1,\ldots,r_d}^{y})$ satisfies
  \[
    h^*(\Delta_{(1,\bq ')};z) = \left(\sum_{t=0}^{\ell+y-1}z^t\right) \cdot g_\br^\bx(z).
  \]
\end{theorem}

\begin{proof}
  Let $M:=\lcm{r_1,\ldots,r_d}$.
  First observe that if $\bx$ is an R-multiplicity of $\br$, then 
  \[
    1+\sum_{i=1}^dx_ir_i+yM = (\ell+y)M
  \]
  and thus $(\bx,y)$ is clearly an R-multiplicity of $(\br,M)$.
  Further, 
  \[
    \lcm{r_1,\ldots,r_d,M}=\lcm{r_1,\ldots,r_d}
  \]
  and thus
  \[
    g_{(\br,M)}^{(\bx,y)}(z):=\sum_{0\leq \alpha < \lcm{r_1,\ldots,r_d}}z^{u(\alpha)}
  \]
  where
  \[
    u_{(\br,M)}^{(\bx,y)}(\alpha) = \alpha (\ell+ y) - \sum_{i=1}^d x_i\left\lfloor  \frac{\alpha}{s_i} \right\rfloor - y\left\lfloor  \frac{\alpha}{1} \right\rfloor =   \alpha (\ell) - \sum_{i=1}^d x_i\left\lfloor  \frac{\alpha}{s_i} \right\rfloor = u_{\br}^{\bx}(\alpha) \, .
  \]
  Hence,
  \[
    g_{(\br,M)}^{(\bx,y)}(z)=g_{\br}^{\bx}(z)
  \]
  and the result follows.
\end{proof}

\subsection{A Useful Form for $g_\br^\bx(z)$}
Our goal in this subsection is to prove Theorem~\ref{thm:gxrthm} below, providing a reformulation of $g_\br^\bx(z)$ that is helpful for establishing factorizations.
We will require the following theorem from elementary number theory.

\begin{theorem}[Generalized Chinese Remainder Theorem] \label{thm:GCRT}
  Suppose $m_1, m_2, \dots, m_d$ are positive integers and $i_1, i_2 \dots, i_d \in \Z.$ Then the system of congruences
  \begin{equation}
    \label{system}
    \begin{cases}
      x \equiv i_1 &\ \mod m_1 \\
      x \equiv i_2 &\ \mod m_2 \\
      \hspace{3mm} \vdots      & \hspace{7mm} \vdots \\
      x \equiv i_d &\ \mod m_d 
    \end{cases}
  \end{equation}
  has a solution if and only if $\gcd(m_j, m_{j'}) \mid (i_j - i_{j'})$ for any pair of indices $(j, j'),$ where $1 \le j < j' \le d.$ Moreover, when there is a solution, it is unique modulo $\lcm{m_1, m_2, \dots, m_d}.$
\end{theorem}

Motivated by the above theorem, for two vectors $\br$ and $\bs$ related by~\eqref{equ:sdefn} we define 
\[ 
  I = I(\br) := \{ \bi = (i_1, \dots, i_d) \in \res{s_1} \times \cdots \times \res{s_d} \ : \ \gcd(s_j, s_{j'}) \mid (i_j - i_{j'}) \text{ for all } 1 \le j < j' \le d \} \, .
\]
The following result is a direct consequence of Theorem \ref{thm:GCRT} and~\eqref{equ:lcmeq}.

\begin{corollary} \label{cor:GCRT}
  For each $\bi \in I(\br)$, there exists a unique $\alpha \in \res{\lcm{r_1,\dots, r_d}}$ such that $\alpha \equiv i_j\bmod s_j$ for each $1 \le j \le d.$
\end{corollary}

\begin{definition}\label{defn:omega}
  We denote by $\alpha(\bi)$ the unique $\alpha$ assumed by the above corollary, and let 
  \begin{equation}
    \omega_j = \omega_j(\bi) := \left\lfloor \frac{\alpha(\bi)}{s_j} \right\rfloor.	
    \label{equ:omegadefn}
  \end{equation}
  Thus,
  \begin{equation}
    \alpha(\bi) = \omega_j(\bi) \cdot s_j + i_j \, .  
    \label{equ:alphabi}
  \end{equation}
\end{definition}

The following theorem provides an expression for $g_\br^\bx(z)$ that we will rely on throughout the remainder of this work.

\begin{theorem}\label{thm:gxrthm}
  Assume Setup \ref{setup1}.
  Suppose $(\bc, \brho)$ is a desirable $\bs$-division of $\bx.$ Then
  \[ 
    g_\br^\bx(z) = \sum_{\bi \in I(\br)} z^{\sum_{j=1}^d (c_j i_j - \rho_j \omega_j(\bi))} \, .
  \]
\end{theorem}
\begin{proof}
  By Definition~\ref{def:ellpoly} and Corollary \ref{cor:GCRT}, it is enough to verify that for each $\bi \in I(\br),$ we have 
  \begin{equation}
    u(\alpha(\bi))= \alpha(\bi) \ell - \sum_{j=1}^d x_j\left\lfloor  \frac{\alpha(\bi)}{s_j} \right\rfloor=\sum_{j=1}^d (c_j i_j - \rho_j \omega_j(\bi)).
    \label{equ:ualphabi}
  \end{equation}
  However, it is straightforward to show this by using \eqref{equ:divx}, \eqref{equ:omegadefn}, \eqref{equ:alphabi}, and Lemma \ref{lem:desell}.
\end{proof}

In the case where $\br$ and $\bs$ are related by~\eqref{equ:sdefn} with the entries of $\bs$ pairwise coprime, the following proposition provides an alternative description of $\omega_j$, and hence of $g_\br^\bx(z)$.
Recall that $(a \bmod b)$ is the unique integer $a' \in \res{b}$ satisfying $a \equiv a' \pmod{b}.$
\begin{prop}\label{prop:formula_for_alpha_omega}
  Assume Setup \ref{setup1} where $s_1, \dots, s_d$ are pairwise coprime. Then 
  \[ \lcm{r_1, \dots, r_d} = \lcm{s_1, \dots, s_d} = s_1 s_2 \dots s_d\]
  and thus for each $1 \le j \le d,$ we have $\ds r_j = \prod_{j' \neq j} s_{j'}.$

  Suppose $(\bc, \brho)$ is an $\bs$-division of $\bx.$
  Then for each $\bi \in I(\br),$ 
  \begin{equation}\label{equ:alpha}
    \alpha(\bi) = \left(-\sum_{t=1}^d \rho_t r_t i_t \bmod s_1s_2\dots s_d\right).
  \end{equation}
  Furthermore, if $\brho$ is desirable, then for each $1 \le j \le d,$ 
  \begin{equation}
    \omega_j(\bi) = \left( \sum_{t \neq j} \rho_{t} \frac{r_{t}}{s_j} (i_{j}-i_t) \bmod r_j \right).
    \label{equ:omega}
  \end{equation}
\end{prop}

\begin{proof}
  It is straightforward to verify the conclusions in the first paragraph.

  By the definition of $\alpha(\bi)$ and because the $s_j$'s are pairwise coprime, in order to show \eqref{equ:alpha} it is enough to prove that for each $1 \le j \le d,$
  \begin{equation}
    -\sum_{t=1}^d \rho_{t} r_{t} i_{t}  \equiv i_j \pmod{s_j}.
    \label{equ:alpha1}
  \end{equation}
  However, since $r_t = \prod_{j' \neq t} s_{j'},$ clearly $s_j$ divides $r_t$ for each $t \neq j.$ Hence, 
  $\ds -\sum_{t=1}^d \rho_{t} r_{t} i_{t}  \equiv -\rho_j r_j i_j \pmod{s_j}.$ Next, it follows from Lemma \ref{lem:desrem} that $\ds \sum_{j=1}^d \rho_t r_t \equiv -1 \pmod{s_j}.$ Again, as $s_j$ divides $r_t$ whenever $t \neq j,$ we conclude that $\rho_j r_j \equiv -1 \pmod{s_j}.$ Thus, \eqref{equ:alpha1} follows.

  By the definition of $\omega_j(\bi),$ we see that $\omega_j(\bi) \in \res{r_j}.$ Hence, \eqref{equ:omega} is equivalent to
  \begin{equation}
    \omega_j(\bi) \equiv \sum_{t \neq j} \rho_{t} \frac{r_{t}}{s_j} (i_j-i_t) \pmod{r_j}, 
    \label{equ:omega1}
  \end{equation}
  By \eqref{equ:alpha}, we have that $\alpha(\bi) = -\sum_{t=1}^d \rho_t r_t i_t + M s_1 s_2 \dots s_d = -\sum_{t=1}^d \rho_t r_t i_t + M s_j r_j$ for some integer $M.$ Hence,
  \[\omega_j(\bi) = \frac{\alpha(\bi)- i_j}{s_j} \equiv  \frac{ -\sum_{t=1}^d \rho_t r_t i_t - i_j }{s_j} \pmod{r_j}.\]
  Since $\brho$ is desirable, $\sum_{t=1}^d \rho_t r_t = -1.$ Hence, we can replace $-i_j$ with $\sum_{t=1}^d \rho_t r_t i_j$ in the above equation, from which \eqref{equ:omega1} follows.
\end{proof}


\section{Some Kronecker $h^*$-Polynomials When $\br=(a,ka-1)$}\label{sec:twointegers}

We have seen in Proposition~\ref{prop:oneinteger} that any reflexive $\Pq$ supported on one integer has $\br=(1)$.
The next level of complexity of $\bq$-vectors are those for which $\bq$ has two distinct entries.
Payne's simplices from Example~\ref{ex:payne} are an important example of this type in Ehrhart theory, as they are reflexive polytopes whose $h^*$-polynomials are not unimodal; further, their $h^*$-polynomials factor as a product of geometric series.
In this section we prove four theorems establishing Kronecker $h^*$-polynomials, each theorem corresponding to a family of $\bq$-vectors supported on two integers.
We use the following setup throughout this section.

\subsection{Setup} \label{subsec:setup}
Recall from elementary number theory that for $\br=(r_1,r_2) \in \left( \Z_{>0} \right)^2$ such that $\gcd(r_1,r_2) = 1$, there exists an integer solution $\brho = (\rho_1, \rho_2)$ to $\rho_1 r_1 + \rho_2 r_2 = -1$. Furthermore, if $\brho^* = (\rho_1^*, \rho_2^*)$ is a special integer solution to $\rho_1 r_1 + \rho_2 r_2 = -1$, then all integer solutions are in the form of
\[ \rho_1 = \rho_1^* - r_2 k, \quad \rho_2 = \rho_2^* + r_1 k, \quad \text{for some integer $k$.}\]
It then follows that there exists a unique integer solution $\brho = (\rho_1, \rho_2)$ to $\rho_1 r_1 + \rho_2 r_2 = -1$ where $\rho_1 \in [-r_2]$ and $\rho_2 \in \res{r_1}$.
This implies that desirable $\bs$-remainders are unique in this context.

\begin{setup}\label{setup2} 
  Let $\br=(r_1,r_2) \in  (\Z_{>0})^2$ satisfy $\gcd(r_1,r_2) = 1$, and let $\bs = (s_1, s_2) = (r_2, r_1).$ 
  Let $\brho = (\rho_1, \rho_2)$ be the unique solution to $\rho_1 r_1 + \rho_2 r_2 = -1$ such that $\rho_1 \in [-s_1]$ and $\rho_2 \in \res{s_2}$.
  Let $\bq$ be the vector supported by $\br$ with the R-multiplicity $\bx= (x_1, x_2) \in \left( \Z_{>0} \right)^2$ having the property that $\brho$ is an $\bs$-remainder of $\bx$; that is, for some integers $c_1, c_2,$
  \[ 
    x_1 = c_1 s_1 + \rho_1 \text{ and } x_2 = c_2 s_2 + \rho_2 \, .
  \]
  Thus, $\ell =\ell(\bq)= c_1 + c_2.$
\end{setup}

\begin{example}\label{ex:uniquerho}
	Suppose $\br = (a, ka-1)$ for some integers $a \ge 2$ and $k \ge 1.$ 
  Then $\bs = (ka-1, a),$ $\brho=(-k, 1)$, and $\bx = (c_1(ka-1)-k, c_2 a +1)$ for some integers $c_1 > k/(ka-1)$ and $c_2 \geq 0.$
\end{example}

Since $\gcd(s_1, s_2) = \gcd(r_1, r_2)  = 1$ for this setup, we can always apply Proposition \ref{prop:formula_for_alpha_omega}, yielding the following corollary.
\begin{corollary}\label{cor:alpha_omega_a_ka-1}
  Assume Setup \ref{setup2}. Then for each $\bi = (i_1, i_2) \in I(\br) = \res{r_2} \times \res{r_1},$ we have:
  \begin{align*}
    \alpha(\bi) =& \left( -\rho_1 r_1 i_1 - \rho_2 r_2 i_2 \bmod r_1 r_2 \right) \\
    \omega_1(\bi) =& \left( \rho_2 (i_1 - i_2) \bmod r_1 \right)  \\
    \omega_2(\bi) =& \left( \rho_1 (i_2 - i_1) \bmod r_2 \right) 
  \end{align*}
\end{corollary}

The following lemma will be used in the proofs of the theorems in the next subsection.

\begin{lemma}\label{lem:g2supp-ka-1}
  Let $\br=(a,ka-1)$ for some $k\geq 1$ and $a\geq 2$.
  Let $(\bc, \brho)$ be the desirable $\bs$-division of $\bx$ with $\brho = (-k, 1).$ Then
  \[ g_\br^\bx(z) = \sum_{\bi \in \res{ka-1} \times \res{a}} z^{c_1 i_1 + c_2 i_2 - \left\lfloor\frac{i_1-i_2}{a} \right\rfloor}.\]
\end{lemma}

\begin{proof} 
  Let $M = \ds \left\lfloor \frac{i_1-i_2}{a} \right\rfloor.$
  Note that $I(\br) = \res{r_2} \times \res{r_1} = \res{ka-1} \times \res{a}.$ 
  Hence, we only need to show that, using the notation from Definition~\ref{defn:omega},
  \[
    -k \omega_1(\bi) + \omega_2(\bi) = \left\lfloor \frac{i_1-i_2}{a} \right\rfloor = M \, ,
  \]
  and the result follows from Theorem~\ref{thm:gxrthm}. 
  Applying Corollary \ref{cor:alpha_omega_a_ka-1}, we get
  \[ \omega_1(\bi) = \left( (i_1 - i_2) \bmod a \right) \quad \text{and} \quad \omega_2(\bi) = \left( k (i_1 - i_2) \bmod (ka-1) \right).\]
  Thus, $i_1 - i_2 = a M +\omega_1(\bi)$ and
  \begin{equation}\label{eqn:alphaomega}
    \omega_2(\bi) =  \left( k (i_1 - i_2) \bmod (ka-1) \right) = \left(M + k \omega_1(\bi) \bmod (ka-1) \right) \, .
  \end{equation}
  Since $(i_1,i_2) \in \res{ka-1} \times \res{a},$ we have that $-(a-1) \le i_1 - i_2 \le (ka-2).$ 
  The proof is complete after we show that the right hand side of~\eqref{eqn:alphaomega} is equal to $M + k \omega_1(\bi),$ which is equivalent to 
  \[
    0 \le M + k \omega_1(\bi) \le ka-2 \, .
  \]
  It is straightforward to verify the left-hand inequality 
  \[
    0 \le M + k \omega_1(\bi)=\left\lfloor \frac{i_1-i_2}{a} \right\rfloor+ k \left( (i_1 - i_2) \bmod a \right)
  \]
  holds by considering the two cases $i_1 -i_2 < 0$ and $i_1 -i_2 \ge 0$, noting the assumption that $k\geq 1$.
  One can similarly verify the right-hand inequality holds by considering the two cases $i_1-i_2 < (k-1)a$ and $i_1-i_2 \ge (k-1)a.$
  \commentout{
    Note that
    \[ \alpha(\bi) = \omega_1(\bi) s_1 + i_1 = \omega_2(\bi) s_2 + i_2 \in \res{\lcm{r_1, r_2}}.\]
    In this case, we have
    \[ \omega_1(\bi) (ka-1) + i_1 = \omega_2(\bi) a + i_2 \in \res{ a(ka-1)}.\]
    Thus, 
    \[  -k \omega_1(\bi) + \omega_2(\bi) = \frac{i_1-i_2 - \omega_1(\bi)}{a}.\]
    Then the conclusion follows from the fact that $\omega_1(\bi) \in \res{r_1}= \res{a}$ and $k \omega_1(\bi) - \omega_2(\bi)$ is an integer.
  }
\end{proof}

\subsection{Four Main Theorems}

\begin{theorem}\label{thm:case0}
  For $\br=(1,a)$ or $(a,1)$ and any R-multiplicity $\bx$, the resulting $g_\br^\bx(z)$ is a geometric series, which is a Kronecker polynomial.
\end{theorem}

\begin{proof}
  Suppose $\br = (1, a)$ for some integer $a \ge 2.$ 
  Then $\bs =(a, 1),$ $\brho = (-1, 0),$ and $\bx = (ac_1 -1, c_2)$ for some positive integers $c_1, c_2.$
  Then $\omega_1(\bi) = \left( \rho_2 (i_1 - i_2) \bmod r_1 \right) = \left( 0 \bmod r_1 \right) = 0.$ 
  Thus,
  \[ 
    \rho_1 \omega_1(\bi) + \rho_2 \omega_2(\bi) = -1 \cdot 0 + 0 \cdot \omega_2(\bi) = 0 \, .
  \]
  Hence, 
  \[
    g_{\br}^{\bx}(z) = \sum_{\bi \in \res{a} \times \res{1} } z^{c_1 i_1 + c_2 i_2} = \sum_{i_1 \in \res{a}} z^{c_1 i_1} 
  \]
  is a Kronecker polynomial.

  The proof in the case where $\br=(a,1)$ for some integer $a\geq 2$ is identical.
\end{proof}

\begin{theorem}\label{thm:case1}
  Let $a\geq 2$, $k\geq 1$, and $c\geq 1$.
  For $\br = (a, ka-1)$ and $\bx = ( (ka-1)c-k, a ( (ka-1)c-k) + 1),$ we have
  \[ g_\br^\bx(z) = \left(\sum_{j_1 \in \res{ka-1}} z^{(ac-1) j_1} \right) \left(\sum_{j_2 \in \res{a}} z^{cj_2} \right),\]
  which is a Kronecker polynomial.
\end{theorem}

\begin{proof}
  With the given $\bx,$ we have the desirable $\bs$-division with
  \[ \bc = (c, (ka-1)c-k) \text{ and } \brho = (-k, 1).\]
  Observe that $r_1 = s_2 = a$ and $r_2 = s_1 = ka-1$.
  Thus, by Lemma~\ref{lem:g2supp-ka-1},
  \[ g_\br^\bx(z) = \sum_{\bi \in \res{ka-1} \times \res{a}} z^{c i_1 + ( (ka-1)c-k) i_2 - \left\lfloor\frac{i_1-i_2}{a} \right\rfloor}.\]	
  One sees that it is enough to show that there exists a bijection $\varphi$ on $\res{ka-1} \times \res{a}$ such that for any $\bi=(i_1,i_2) \in \res{ka-1} \times \res{a},$ if $\bj = (j_1, j_2) = \varphi(\bi)$, then
  \begin{equation}
    c i_1 + ( (ka-1)c-k) i_2 - \left\lfloor\frac{i_1-i_2}{a} \right\rfloor= (ac-1)j_1 + c j_2.  
    \label{equ:u-itoj}
  \end{equation}
  We will construct such a bijection below.

  For any $\bi = (i_1, i_2) \in \res{ka-1} \times \res{a},$ we define
  \[ \varphi_1(\bi) = (ka-1) i_2 + i_1, \text{ and } \varphi_2(\bi) = a i_1 + i_2.\]
  It is easy to see that both $\varphi_1$ and $\varphi_2$ are bijections from $\res{ka-1} \times \res{a}$ to $\res{a(ka-1)}.$ Therefore, $\varphi := \varphi_2^{-1} \circ \varphi_1$ is a bijection on $\res{ka-1} \times \res{a}$.
  Now suppose $\bj = (j_1, j_2) = \varphi(\bi)=\varphi(i_1,i_2).$ By the definition of $\varphi,$ we have
  \[ j_1 = \left\lfloor \frac{(ka-1)i_2 + i_1}{a} \right\rfloor, \text{ and } j_2 = \left((ka-1)i_2 + i_1 \right)- aj_1.\]
  Thus,
  \[ j_1 = ki_2 + \left\lfloor \frac{i_1-i_2}{a} \right\rfloor, \text{ and } j_2 = i_1-i_2 - a \left\lfloor \frac{i_1-i_2}{a}\right\rfloor.\] 
  One then can show \eqref{equ:u-itoj} holds directly by plugging in the above. 
\end{proof}

\begin{theorem}\label{thm:case2}
  Let $a\geq 2$ and $c\geq 1$.
  For $\br = (a, a-1)$ and $\bx = ( (a-1)c-1, ac+1)$, we have
  \[ g_\br^\bx(z) = (1+z^{c+1})\left(\sum_{j=0}^{2\left\lfloor\frac{a-1}{2}\right\rfloor}z^{cj}\right)\left(\sum_{j=0}^{\left\lceil\frac{a-1}{2}\right\rceil -1}z^{2cj}\right), 
  \]
  which is a Kronecker polynomial.
\end{theorem}

\begin{proof}
  With the given $\bx,$ we have the desirable $\bs$-division with
  \[ 
    \bc = (c, c) \text{ and } \brho = (-k, 1) \, .
  \]
  Observe that $r_1 = s_2 = a$ and $r_2 = s_1 = a-1$.
  Thus, by Lemma~\ref{lem:g2supp-ka-1},
  \[ 
    g_\br^\bx(z) = \sum_{\bi \in \res{a-1} \times \res{a}} z^{c (i_1 + i_2) - \left\lfloor\frac{i_1-i_2}{a} \right\rfloor} \, .
  \]
  We have further that $0\leq i_1\leq a-2$ and $0\leq i_2\leq a-1$, and thus
  \begin{equation}\label{eqn:indicatorfunction}
    \left\lfloor\frac{i_1-i_2}{a} \right\rfloor = \left\{ \begin{array}{cc}0 & \text{ if }i_1\geq i_2 \\ -1 & \text{ if }i_1<i_2 \end{array}\right. \, .
  \end{equation}
  Define $A:=\{(i_1,i_2)\in \res{a-1} \times \res{a}:i_1\geq i_2\}$ and $B:=\{(i_1,i_2)\in \res{a-1} \times \res{a}:i_1< i_2\}$.
  We define a bijection $\phi:\res{a-1} \times \res{a}\to \res{a-1} \times \res{a}$ by sending $(i_1,i_2)\in A$ to the element $(i_2,i_1+1)\in B$ and sending the element $(i_1,i_2)\in B$ to the element $(i_2-1,i_1)\in A$.

  Now, using~\eqref{equ:ualphabi} and~\eqref{eqn:indicatorfunction} we see that for $(i_1,i_2)\in A$, we have
  \[
    u(\phi(\alpha(i_1,i_2)))=c(i_2+i_1+1)-(-1)=c(i_1+i_2)+c+1=u(\alpha(i_1,i_2))+c+1 \, .
  \]
  Thus, 
  \[
    g_{\br}^\bx(z)=(1-z^{d+1})\sum_{(i_1,i_2)\in A}z^{c(i_1+i_2)} \, .
  \]
  To complete the proof, it is enough to show that 
  \[ 
    \sum_{0 \le i_2\leq i_1 \leq a-2} z^{c(i_1+i_2)} = \left(\sum_{j=0}^{2\left\lfloor\frac{a-1}{2}\right\rfloor}z^{cj}\right)\left(\sum_{j=0}^{\left\lceil\frac{a-1}{2}\right\rceil -1}z^{2cj}\right) \, , 
  \]
  which is a straightforward exercise using induction on $a$. 
\end{proof}

\begin{theorem}\label{thm:case3}
  Let $a\geq 2$ and $c\geq 1$.
  For $\br = (a, a^2-1)$ and $\bx = ( (a^2-1)c-a, a(ac-1)+1)$, we have
  \[ g_\br^\bx(z) = \left(\sum_{j_1 \in \res{a} } z^{(ac-1)j_1} \right)\left(\sum_{j_2 \in \res{a+1} } z^{cj_2}\right)\left(\sum_{j_3 \in \res{a-1} } z^{ (ac+c-1) j_3} \right), 
  \]
  which is a Kronecker polynomial.
\end{theorem}

\begin{proof}
  With the given $\bx,$ we have the desirable $\bs$-division with
  \[ 
    \bc = (c, ac-1) \text{ and } \brho = (-a, 1) \, .
  \]
  Observe that $r_1 = s_2 = a$ and $r_2 = s_1 = a^2-1$. For convenience, for any $\bi \in \Z^2,$ we let
  \[ u(\bi) := c i_1 + (ac-1) i_2 - \left\lfloor\frac{i_1-i_2}{a} \right\rfloor.\]
  It is straightforward to verify that for any $m=1,2,\dots, a-1,$ we have
  \begin{equation}\label{equ:change_i_1}
    u( ma-1, a-1) = u (a^2-1, m-1).
  \end{equation}
  Notice that 
  \[ I' := \res{a^2} \times \res{a} \setminus \{ (ma-1,a-1) \ : \ m=1,2,\dots, a\} \]
  is the set obtained from $I(\br) = \res{a^2-1} \times \res{a}$ by replacing each $(ma-1, a - 1)$ with $(a^2-1, m-1)$ for $m=1,2,\dots,a-1.$
  Hence, it follows from Lemma~\ref{lem:g2supp-ka-1} and \eqref{equ:change_i_1} that
  \[ g_\br^\bx(z) = \sum_{\bi \in I(\br)} z^{u(\bi)} = \sum_{\bi \in I'} z^{u(\bi)}.\]
  Next, one sees that if we let $I_0 := \res{a} \times \res{a} \setminus \{(a-1,a-1)\},$ then $I'$ can be decomposed as
  \[ I' = \biguplus_{j_1 \in \res{a}} \{ ( j_1 a + i_1, i_2 ) \ : \ (i_1,i_2) \in I_0 \}.\]
  Since $u(j_1 a + i_1, i_2) = (ac-1)j_1 + u(i_1,i_2),$ we immediately have that
  \[ g_\br^\bx(z) = \left( \sum_{j_1 \in \res{a}} z^{(ac-1)j_1} \right) \left(\sum_{\bi \in I_0} z^{u(\bi)} \right).\]
  Finally, one sees that for each $\bi = (i_1, i_2) \in I_0$, there exists a unique $(j_2, j_3) \in \res{a+1} \times \res{a-1}$ such that
  \[ a i_2 + i_1 = (a+1) j_3 + j_2.\]
  This defines a bijection $\Psi$ from $I_0$ to $\res{a+1} \times \res{a-1}.$ Since 
  \[ a i_2 + i_1 = (a+1) i_2 + (i_1-i_2) = (a+1)(i_2-1) + (a+1+i_1-i_2) \]
  and $-(a-1) \le i_1 - i_2 \le a-1,$ we conclude that if $(j_2,j_3) = \Psi(i_1,i_2),$ then 
  \[ j_2= i_1-i_2 - (a+1) \left\lfloor\frac{i_1-i_2}{a}\right\rfloor, \quad j_3 = i_2 +  \left\lfloor\frac{i_1-i_2}{a}\right\rfloor.\]
  Using the above, it is easy to verify that
  \[ c j_2 + (ac+c-1) j_3 = u(\bi) =c i_1 + (ac-1) i_2 - \left\lfloor\frac{i_1-i_2}{a} \right\rfloor.\]
  Then our conclusion follows.
\end{proof}


\section{A Classification When $\br=(2,2k-1)$}\label{sec:twoodd}

Given the positive results in Section~\ref{sec:twointegers}, it is natural to ask if it is possible to classify those $(\br,\bx)$ such that $g_\br^\bx(z)$ admits a geometric factorization.
In this section, we prove Theorem~\ref{thm:2odd}, providing a first step in response to this question.
We will work in the context of the following setup.

\subsection{Setup and Classification}
\begin{setup} \label{setup3}
  Let $\br = (2, 2k-1)$ for some integer $k \ge 2.$ Then $\brho = (-k, 1)$ and $\bx = ( (2k-1)c_1 - k, 2 c_2 + 1)$ for some integers $c_1 \ge 1$ and $c_2 \ge 0.$
  Applying Lemma \ref{lem:g2supp-ka-1}, we have that
  \begin{align}\label{equ:2-2k-1-g}
    g_\br^\bx(z)  & = \sum_{\bi \in \res{2k-1} \times \res{2}} z^{c_1 i_1 + c_2 i_2 - \left\lfloor\frac{i_1-i_2}{2} \right\rfloor} \nonumber \\
                  & = 
                    \left( \begin{array}{l}
                             z^0  + z^{c_1}  + z^{2c_1 -1}  + z^{3c_1 - 1}  + \cdots + \\
                             \hspace{.7cm}  z^{(2k-3) c_1 - (k-2)}  + z^{(2k-2) c_1 - (k-1)} + \\
                             \hspace{1.4cm} z^{c_2+1}  + z^{c_1 + c_2}  + z^{2c_1 + c_2}  +z^{3c_1 + c_2 -1}  + \cdots + \\
                             \hspace{2.1cm}  z^{(2k-3)c_1 + c_2 - (k-2)}  + z^{(2k-2)c_1 + c_2 - (k-2)} 
                           \end{array}\right)
  \end{align}
  where the first two lines of~\eqref{equ:2-2k-1-g} correspond to summands with $i_2=0$ and the last two lines of~\eqref{equ:2-2k-1-g} correspond to summands with $i_2=1$.
  Suppose further in our setup that if $g_\br^\bx(z)$ has a geometric factorization, it is given as follows for some $\gamma_1,\ldots,\gamma_p\geq 2$ and $e_1 \le e_2 \le \dots \le e_p$.
  \begin{equation} \label{equ:g-geofac}
    g_\br^\bx(z) = \prod_{j=1}^p \sum_{i=0}^{\gamma_j-1} z^{i e_j} = \prod_{j=1}^p \left( 1 + z^{e_j} + z^{2 e_j} + \cdots z^{(\gamma_j-1) e_j} \right)
  \end{equation}
\end{setup}

Our main result of this section is the following.

\begin{theorem}\label{thm:2odd}
  Suppose $\br = (2, 2k-1)$ for some integer $k \ge 2$. 
  Then $g_\br^\bx(z)$ has a geometric factorization if and only if $(\br, \bx) = ((2,9), (4,3))$ or $(\br,\bx)$ is one of the cases given by Theorems \ref{thm:case1}, \ref{thm:case2}, and \ref{thm:case3}.
  Specifically, assume Setup \ref{setup3} holds and $g_\br^\bx(z)$ admits a geometric factorization. Then $c_1 \neq c_2 + 1$ and two cases arise:
  \begin{enumerate} 
  \item Suppose $c_2 + 1 < c_1.$ 
    \begin{enumerate}
    \item If $c_1 = 2(c_2+1),$ then $\br=(2,3)$ and $c_2$ can be any non-negative integer, which corresponds to applying Theorem \ref{thm:case1} with $a=3$ and $k=1$ to obtain $\bx = (6c-2, 2c-1)$ for $c \ge 1.$
    \item If $c_1 \neq 2(c_2+1)$, then $\br=(2,3)$ and $c_2 = c_1 -2,$ which corresponds to applying Theorem \ref{thm:case2} with $a=3$ to obtain $\bx = (3c+1, 2c-1)$ for $c \ge 2.$
    \end{enumerate}
  \item Suppose $c_1 < c_2 + 1.$
    \begin{enumerate} 	\item If $c_2 + 1 = 2c_1,$ then either $\br=(2,9)$ and $c_1=1$ (so $\bx=(4,3)$), or $\br=(2,3)$ and $c_1$ can be any positive integer. Note that the latter situation corresponds to applying Theorem \ref{thm:case3} with $a=2$ to obtain $\bx = (3c-2, 4c-1)$ for $c \ge 1.$
    \item If $c_2 + 1 \neq 2 c_1$, then $c_2 = (2k-1) c_1 - k,$ which corresponds to cases given by Theorem \ref{thm:case1} with $a=2$. 
    \end{enumerate}
  \end{enumerate}
\end{theorem}

Our proof will require the following two lemmas.
Recall that $[z^t]f(z)$ denotes the coefficient of $z^t$ in $f(z)$.

\begin{lemma}\label{lem:basic-geofac}
  Suppose $f(z)$ has a geometric factorization as given in \eqref{equ:geofac}. Assume $e_1 \le e_2 \le \dots \le e_p$ and express $f(z)$ as
  \begin{equation}
    f(z) = 1 + z^{\mu_1} + z^{\mu_2} + \cdots + z^{\mu_M} \quad \text{ with $0 < \mu_1 \le \mu_2 \le \cdots \le \mu_M$.}
    \label{equ:f(z)_termbyterm}
  \end{equation} 
  Then the following are true.
  \begin{enumerate}[(i)]
  \item \label{itm:e1mu1} $e_1 = \mu_1$.
    Furthermore, if $[z^{e_1}] f = m,$ then $e_1 = e_2 = \dots = e_m \neq e_{m+1}.$ 
  \item \label{itm:e2mu2} If $\mu_2 \neq 2 \mu_1,$ then $e_2 = \mu_2.$
  \item \label{itm:musum} If $z^{\mu_1 + \mu_2}$ does not appear in \eqref{equ:f(z)_termbyterm}, then $\mu_2 = 2 \mu_1$ and $\gamma_1 = 3.$ So $(1 + z^{\mu_1} + z^{2\mu_1})$ is a factor in the geometric factorization \eqref{equ:geofac} of $f(z).$

  \item \label{itm:mtp_mu1} For any $i \in \{2,3,\dots,M\},$ if $\mu_i$ cannot be written as a non-negative integer linear combination of $\mu_{1}, \dots, \mu_{j-1},$ then $\mu_i = e_j$ for some $j.$  
    In particular, if $\mu_i$ is not a multiple of $\mu_1,$ but $\mu_{i'}$ is a multiple of $\mu_1$ for every $1 \le i' < i$, then $\mu_i = e_j$ for some $j.$
  \item \label{itm:compcoeff} For any subset $S \subseteq \{1, 2, \dots, p \},$ and any $t \in \Z_{\ge 0},$ 
    \[ [z^t] f(z) \ge [z^t] \left(\prod_{j \in S} \sum_{i=0}^{\gamma_j-1} z^{i e_j} \right).\]

  \item \label{itm:sumcoeff} For any exponent $e_j$ of the factorization and any $e \ge e_j,$ we have 
    \[ [z^{e-e_j}] f + [z^{e+e_j}] f \ge [z^{e}] f.\]
  \end{enumerate}
\end{lemma}

\begin{proof}
  We omit the proof for all but parts~\eqref{itm:musum} and~\eqref{itm:sumcoeff}, as the others are straightforward exercises from the definition.
  For part~\eqref{itm:musum}, if $z^{\mu_1 + \mu_2}$ does not appear in~\eqref{equ:f(z)_termbyterm}, then we must have $\mu_2\neq e_2$.
  Hence, by the contrapositive of part~\eqref{itm:e2mu2}, $\mu_2=2\mu_1$.
  Since we assumed that $3\mu_1=\mu_2+\mu_1$ is not an exponent in~\eqref{equ:f(z)_termbyterm}, then $\gamma_1=3$, and we have our desired factor.

  For part~\eqref{itm:sumcoeff}, if $e$ is written as a non-negative integer linear combination $\mathcal{C}$ of $e_1,\ldots,e_p$ using less than $\gamma_j-1$ $e_j$'s, then $\mathcal{C}+e_j$ contributes an exponent in~\eqref{equ:f(z)_termbyterm}.
  If $e$ can only be written as a non-negative integer linear combination $\mathcal{C}$ of $e_1,\ldots,e_p$ using all of the $\gamma_j-1$ $e_j$'s, then $\mathcal{C}-e_j$ contributes an exponent in~\eqref{equ:f(z)_termbyterm}.
  Thus, for each non-negative integer linear combination $\mathcal{C}$ giving $e$, we obtain at least one combination giving either $e-e_j$ or $e+e_j$.
\end{proof}

\begin{lemma}\label{lem:basic_properties_22k-1}
  If Setup \ref{setup3} holds and $g_\br^\bx(z)$ admits a geometric factorization, then the following are true.
  \begin{enumerate}[(i)]
  \item \label{itm:gamma_prod} $2(2k-1) = \prod_{j=1}^p \gamma_j.$ Thus, exactly one of $\gamma_j$'s is even.
  \item \label{itm:c1neqc2+1} $c_1 \neq c_2 + 1.$
  \item \label{itm:11} If $(c_1, c_2) = (1,1),$ then $k=2$ or $5,$ that is, $\br=(2,3)$ or $(2,9).$
  \end{enumerate}
\end{lemma}

\begin{proof}
  \begin{enumerate}[(i)]
  \item 
    Comparing the number of monomials in equations~\eqref{equ:2-2k-1-g} and~\eqref{equ:g-geofac}, the result follows.

  \item 
    Assume the contrary that $c_1 = c_2 + 1.$ Then \eqref{equ:2-2k-1-g} becomes
    \[g_\br^\bx(z) = \begin{matrix} z^0 &+ z^{c_1} &+ z^{2c_1 -1} &+ z^{3c_1 - 1} &+ \cdots + &z^{(2k-3) c_1 - (k-2)} &+ z^{(2k-2) c_1 - (k-1)} + \\
        z^{c_1} &+ z^{2c_1 -1} &+ z^{3c_1 -1} &+ z^{4c_1 -2} &+ \cdots + &z^{(2k-2)c_1  - (k-1)} &+ z^{(2k-1)c_1  - (k-1)}.
      \end{matrix}	
    \]
    We consider two cases. If $c_1 = 1,$ then by Lemma \ref{lem:basic-geofac} part \eqref{itm:e1mu1}, we have $e_1=e_2 = e_3 =e_4 = 1.$ This implies that $[z^2] g_\br^\bx(z) \ge \binom{4}{2} = 6.$ However, one sees that the expression above contains at most $4$ copies $z^2,$ a contradiction. If $c_1 > 1,$ then by Lemma \ref{lem:basic-geofac} part \eqref{itm:e1mu1} again, we have $e_1=e_2= c_1.$ It then follows from Lemma \ref{lem:basic-geofac} part \eqref{itm:compcoeff} that $[z^{2c_1}] g_\br^\bx(z) \ge 1,$ contradicting with the fact that $z^{2c_1}$ does not appear in the expression above. Therefore, we must have that $c_1 \neq c_2 + 1.$

  \item 
    It is easy to verify the following: 
    \begin{itemize}
    \item when $\br=(2,3)$, $g_\br^\bx(z)$ has a geometric factorization $(1+z)(1+z+z^2)$,
    \item when $\br=(2,9),$ $g_\br^\bx(z)$ has a geometric factorization $(1+z+z^2)(1+z+z^2)(1+z^2)$,
    \item when $\br=(2,5)$ or $(2,7),$ $g_\br^\bx(z)$ does not have a geometric factorization.
    \end{itemize}

    Now assume $k \ge 6.$ (We will find a contradiction.) Then using~\eqref{equ:2-2k-1-g} we have 
    \[ 
      g_\br^\bx(z) = 1 + 2z+ 4z^2 + 4z^3 + 4z^4 + 4z^5 + c z^6 + z^7 f(z) \, ,
    \]
    where $f(z) \in \Z_{\ge 0}[z]$ and $c=2$ or $4.$
    It follows from Lemma~\ref{lem:basic-geofac} part~\eqref{itm:e1mu1} that $e_1 = e_2 = 1 \neq e_3.$ 

    It follows from part~\eqref{itm:gamma_prod} that one of $\gamma_1$ and $\gamma_2$ is not $2.$
    We next show that both $\gamma_1$ and $\gamma_2$ are not $2.$ Suppose one of them is $2.$ 
    Without loss of generality (due to $e_1 = e_2$), assume $\gamma_1 = 2.$ 
    Then $\gamma_2 \ge 3.$
    Thus,
    \[ \prod_{j=3}^p \sum_{i=0}^{\gamma_j-1} z^{i e_j} = g_\br^\bx(z)/((1+z)(1+z+z^2+ \dots + z^{\gamma_2-1})) = 1 + 2 z^2 + z^3 h(z),\]
    for some polynomial $h(z)$. 
    Thus, by Lemma \ref{lem:basic-geofac} part~\eqref{itm:e1mu1} again, we conclude that $e_3=e_4 = 2.$ However, 
    \[ [z^3] \left(\prod_{j =1}^4 \sum_{i=0}^{\gamma_j-1} z^{i e_j}\right) \ge [z^3] \left( (1+z)(1+z+z^2)(1+z^2)(1+z^2) \right) = 5 > 4 = [z^3] g_\br^\bx(z),\]
    contradicting Lemma \ref{lem:basic-geofac} part \eqref{itm:compcoeff}. Therefore, $\gamma_1 \ge 3.$

    Now given $\gamma_1 \ge 3$ and $\gamma_2 \ge 3,$ we can show $e_3 = 2$ using similar arguments as above. Then one checks that 
    \[ [z^4] \left(\prod_{j =1}^3 \sum_{i=0}^{\gamma_j-1} z^{i e_j}\right) \ge [z^4] \left((1+z+z^2)(1+z+z^2)(1+z^2)\right) =  4 = [z^4] g_\br^\bx(z),\]
    where the equality in ``$\ge$'' holds if and only if $(\gamma_1, \gamma_2, \gamma_3)=(3,3,2).$ Hence, by Lemma \ref{lem:basic-geofac} part \eqref{itm:compcoeff}, we must have $(\gamma_1, \gamma_2, \gamma_3)=(3,3,2).$
    Let $g_0(z) = \prod_{j=4}^p \sum_{i=0}^{\gamma_j-1} z^{i e_j} = g_\br^\bx(z)/((1+z)(1+z+z^2)(1+z+z^2)(1+z^2)).$ Then
    \[ g_\br^\bx(z) = (1+z+z^2)(1+z+z^2)(1+z^2) g_0(z).\]
    By comparing the coefficients of $z^5$ on both sides, we must have that $[z^5] g_0(z) = 2.$ But this implies that 
    \[ [z^6] \left( (1+z+z^2)(1+z+z^2)(1+z^2) g_0(z)\right) \ge 5 > 4\]
    contradicting with the assumption that $[z^6] g_\br^\bx(z) = 2$ or $4.$
    \qedhere
  \end{enumerate}
\end{proof}

\subsection{Proof of Theorem~\ref{thm:2odd}}
Note that Lemma \ref{lem:basic_properties_22k-1} part \eqref{itm:c1neqc2+1} provides the assertion that $c_1 \neq c_2 +1.$ In the proof of Lemma \ref{lem:basic_properties_22k-1} part \eqref{itm:11}, we showed that if $(\br, \bx) = ((2,9), (4,3))$, $g_\br^\bx(z)$ has a geometric factorization. This, together with, Theorems \ref{thm:case1}, \ref{thm:case2}, and \ref{thm:case3}, provides one direction for the if and only if condition in Theorem~\ref{thm:2odd}.
We providing separate proofs of the other direction for parts (1), (2a), and (2b) of Theorem~\ref{thm:2odd}.

\begin{proof}[Proof of Part (1) of Theorem~\ref{thm:2odd}]

  Since $c_2 + 1 < c_1,$ we have $c_1 \ge 2.$
  Express $g_\br^\bx(z)$ as
  \begin{equation}
    g_\br^\bx(z) = 1 + z^{\mu_1} + z^{\mu_2} + \cdots + z^{\mu_M} \quad \text{ with $0 < \mu_1 \le \mu_2 \le \cdots \le \mu_M$.}
    \label{equ:g_termbyterm}
  \end{equation}
  Then by \eqref{equ:2-2k-1-g}, $\mu_1 = c_2+1$ and $\mu_2 = c_1.$ 
Hence, by Lemma \ref{lem:basic-geofac} part \eqref{itm:e1mu1}, $e_1 = \mu_1 = c_2 +1.$
  \begin{enumerate}[(a)]
  \item Suppose $c_1 = 2(c_2 +1).$ Let $c = c_2 +1.$ 
    Then 
    \[ \mu_1 = c, \mu_2 = 2c, \mu_3= 3c - 1, \mu_4 = 4c-1, \mu_5 = 5c-1,\]
    and if $k \ge 3,$
    \[ \mu_6 = 6c-1, \mu_7=7c-2, \mu_8 = 8c-2, \mu_9 = 9c-2.\]
    Hence $z^{\mu_1 + \mu_2}=z^{3c}$ does not appear in $g_\br^\bx(z).$ 
Thus, it follows from part \eqref{itm:musum} of Lemma \ref{lem:basic-geofac} that $(1 + z^c + z^{2c})$ is a factor of given geometric factorization of $g_\br^\bx(z).$ 
Next, one sees that by Lemma \ref{lem:basic-geofac} part \eqref{itm:mtp_mu1} we must have that $e_2 = \mu_3 = 3c-1.$ Given $z^{2 (3c-1)}$ does not appear in $g_\br^\bx(z),$ we conclude that $\gamma_2 = 2.$ Hence,
    \[ (1+z^c + z^{2c})(1 + z^{3c-1}) = 1 + z^c + z^{2c} + z^{3c-1} + z^{4c-1} + z^{5c-1} \]
    appears in the geometric factorization of $g_\br^\bx(z).$ 
    If $k=2,$ i.e., $\br=(2,3),$ the above expression is exactly the geometric factorization of $g_\br^\bx(z).$ 
    If $k \ge 3,$ one can show that $e_3 = 6c-1$ which implies that $z^{c+6c-1} = z^{7c-1}$ appears in \eqref{equ:g_termbyterm}, a contradiction. 

  \item Suppose $c_1 \neq 2 (c_2 + 1),$ so $\mu_2 \neq 2 \mu_1.$ By Lemma \ref{lem:basic-geofac} parts~\eqref{itm:e2mu2} and~\eqref{itm:musum}, $e_2 = \mu_2 = c_1$ and $z^{\mu_1 + \mu_2} = z^{c_1 + c_2 + 1}$ must appear in \eqref{equ:g_termbyterm}.
    However, by looking at Expression \eqref{equ:2-2k-1-g}, we see that the only term that could be $z^{c_1+c_2 +1}$ is $z^{2c_1 -1}.$ 
Hence, $c_1 + c_2 +1 = 2c_1 -1,$ equivalently, $c_2 = c_1 - 2.$ Since $2 = 2 (0 + 1)$ and $c_1 \neq 2 (c_2 + 1),$ we conclude that $c_1 \ge 3.$ Let $c= c_1-1  \ge 2.$ Then
    \[ 
e_1 = \mu_1 = c, e_2 =\mu_2 = c+1, \mu_3 = 2c, \mu_4 = 2c+1, \mu_5 = 3c+1 \, .
\]
    Since $2c+1 < 2c+2 < 3c +1,$ the term $z^{2c+2}$ does not appear in Expression \eqref{equ:g_termbyterm} of $g_\br^\bx(z).$ 
    This implies that $\gamma_2 = 2,$ that is, $(1 + z^{c+1})$ is a factor in the geometric factorization \eqref{equ:g-geofac} of $g_\br^\bx(z).$ 
Then it follows from Lemma \ref{lem:basic_properties_22k-1} part \eqref{itm:gamma_prod} that $\gamma_1$ must be an odd number. 
In particular $\gamma_1 \ge 3.$ 
One sees that $z^{3c}$ does not appear in Expression \eqref{equ:g_termbyterm} of $g_\br^\bx(z)$. Hence, $\gamma_1 =3.$
    Therefore,
    \[ (1 + z^{c} + z^{2c}) (1 + z^{c+1}) = 1 + z^{c} + z^{c+1} + z^{2c} + z^{2c+1} + z^{3c+1} \]
    appears in the geometric factorization of $g_\br^\bx(z).$ Then similarly to part (a), we can show that $\br$ has to be $(2,3).$ \qedhere
  \end{enumerate}
\end{proof}

\begin{proof}[Proof of Part (2a) of Theorem~\ref{thm:2odd}]
  Express $g_\br^\bx(z)$ as \eqref{equ:g_termbyterm}.
  Since $c_1 < c_2 +1,$ one sees that $\mu_1 = c_1.$ 
Hence, by Lemma \ref{lem:basic-geofac} part \eqref{itm:e1mu1}, $e_1 = \mu_1 = c_1.$

  Suppose $c_2 + 1 = 2c_1.$ If $c_1 = 1,$ then $(c_1,c_2)=(1,1)$, and Lemma \ref{lem:basic_properties_22k-1} part~\eqref{itm:11} applies.
Hence, we only need to show that if $c_1 \ge 2$, then $\br=(2,3),$ or equivalently $k = 2.$ 
We prove by contradiction. 
Suppose $c_1 \ge 2$ and $k \ge 3.$
  Let $c= c_1 \ge 2.$ 
  Then 
  \[ 
\mu_1 = c, \mu_2 = 2c-1, \mu_3=2c, \mu_4=  \mu_5 = 3c-1, \mu_6 = 4c-2, \mu_7 = 4c-1,  \mu_8= 5c-2 \, .
\]
  It follows from Lemma \ref{lem:basic-geofac} part~\eqref{itm:e2mu2}, we have $e_2 = \mu_2=2c-1.$ Since $\mu_2 < 2c < \mu_4,$ the term $z^{2c}$ appears exactly once in $g_\br^\bx(z).$ Hence, if $e_3 = 2c,$ we must have that $\gamma_1 = 2$ because $(1+z^c+z^{2c})(1+z^{2c})$ has two copies of $z^{2c}.$ However, in this case
  \[ 
(1+z^c) \left( 1+z^{2c} + \dots + z^{2c(\gamma_3-1)} \right) = \sum_{i=0}^{2\gamma_3 - 1} z^{i c} \, ,
\]
  which is a geometric series with exponent $c$ and of length $2 \gamma_3.$ 
Therefore, we may assume $\gamma_1 \ge 3,$ and $e_3 \neq 2c.$ 
  Now notice that 
  \[ 
\prod_{i=1}^{\gamma_1-1} z^{i c} \prod_{i=1}^{\gamma_2-1} z^{i (2c-1)} = 1 + z^c + z^{2c-1} + z^{2c} + z^{3c-1} + z^{4c-1} + z^{3c} h(z) \, ,
\]
  for some polynomial $h(z),$ and we have previously seen that $[z^{3c-1}]g_\br^\bx(z) = 2.$ 
Thus, we must have that $e_3 = 3c-1.$ 
However, this implies that $z^{4c-1}$ appears in Expression \eqref{equ:g_termbyterm} of $g_\br^\bx(z)$ at least twice as $z^{c} \cdot z^{3c-1}$ and $z^{2c} \cdot z^{2c-1},$ contradicting with the observation that $z^{4c-1}$ only appears once.  
\end{proof}

The proof of (2b) of Theorem~\ref{thm:2odd} is more complated than the other parts, requiring the following lemma.

\begin{lemma}\label{lem:2b}
  Assume that Setup \ref{setup3} holds and $g_\br^\bx(z)$ admits a geometric factorization.
  Suppose further that $c_1 < c_2 +1$ and $2c_1 \neq c_2 + 1$.
  Then $e_1 = c_1.$ Furthermore, $(1 + z^{c_1})$ is a factor in the geometric factorization \eqref{equ:g-geofac} of $g_\br^\bx(z).$ Thus, we may assume $\gamma_1=2.$
\end{lemma}

\begin{proof}
  Since $c_1 < c_2 + 1$ and $2c_1 \neq c_2 +1,$ it is clear that $e_1 = c_1$ and $c_2 \ge 2.$ Assume the contrary that $(1+z^{c_1})$ is not a factor in the geometric factorization \eqref{equ:g-geofac} of $g_\br^\bx(z).$ We consider two cases.

  Suppose $c_1 = 1.$ Then by Lemma \ref{lem:basic-geofac} part~\eqref{itm:e1mu1}, we have $e_1=e_2 = c_1 = 1.$ Since $(1+z)$ is not a factor in the geometric factorization, we have $\gamma_1 \ge 3$ and $\gamma_2 \ge 3.$ 
It follow from Lemma \ref{lem:basic-geofac} part~\eqref{itm:compcoeff} that
  \[ [z^2] g_\br^\bx(z) \ge [z^2]\left( \prod_{j=1}^2 \sum_{i=0}^{\gamma_j-1} z^{i} \right) \ge  [z^2] ((1+z+z^2)(1+z+z^2)) = 3.\]
  However, since $c_2 + 1 \ge 3,$ one sees that there are at most $2$ copies of $z^2$ in Expression \eqref{equ:2-2k-1-g} of $g_\br^\bx(z),$ which is a contradiction.

  Suppose $c_1 \ge 2.$ 
By assumption, we have $\gamma_1 \ge 3.$ 
It then follows that $z^{2c_1}$ appears at least once in $g_\br^\bx(z).$ 
However, the only term in the Expression \eqref{equ:2-2k-1-g} that could be $z^{2c_1}$ is $z^{c_1+c_2}.$ 
Thus, $2c_1 = c_1 + c_2,$ or equivalently, $c_2 = c_1.$
  Then one sees that $c_1+1=c_2+1$ is the second lowest positive order in \eqref{equ:2-2k-1-g}. 
Thus, by Lemma \ref{lem:basic-geofac} part \eqref{itm:e2mu2}, we have $e_2 = c_1 + 1.$ 
It follows that $z^{e_1+e_2} = z^{2c_1+1}$ has to appear in $g_\br^\bx(z).$ However, the only term that could be $z^{2c_1+1}$ is $z^{3c_1-1},$ which implies that $c_1 = 2.$ (So $e_1 = c_2 = c_1 = 2.$)
  Then \eqref{equ:2-2k-1-g} becomes 
  \[ g_\br^\bx(z) =
    \begin{matrix} z^0 &+ z^{2} &+ z^{3} &+ z^{5} &+ \cdots + &z^{3k-4} &+ z^{3k-3} + \\
      z^{3} &+ z^{4} &+ z^{6} &+ z^{7} &+ \cdots + &z^{3k-2} &+ z^{3k}.
    \end{matrix}
  \] 
  Applying Part \eqref{itm:e1mu1} of Lemma \ref{lem:basic-geofac} to $g_\br^\bx(z)/\left( \sum_{i=0}^{\gamma_1-1} z^{2i} \right),$ we obtain that $e_2 = e_3 = 3.$ It then follows from Lemma \ref{lem:basic-geofac} part~\eqref{itm:compcoeff} that
  \[ [z^5] g_\br^\bx(z) \ge [z^5]\left( \prod_{j=1}^3 \sum_{i=0}^{\gamma_j-1} z^{i} \right) \ge [z^5] ((1+z^2)(1+z^3)(1+z^3) = 2,\]
  contradicting the fact that there is as most one copy of $z^5$ in $g_\br^\bx(z).$
\end{proof}

\begin{proof}[Proof of Part (2b) of Theorem~\ref{thm:2odd}]
  By Lemma \ref{lem:2b}, we may assume $e_1=c_1$ and $\gamma_1 =2.$ Let $g(z) = g_\br^\bx(z)/(1+z^{c_1}).$ Then $g(z)$ has a geometric factorization
  \begin{equation} \label{equ:g(z)-geofac}
    g(z) = \prod_{j=2}^p \sum_{i=0}^{\gamma_j-1} z^{i e_j} = \prod_{j=1}^p \left( 1 + z^{e_j} + z^{2 e_j} + \cdots z^{(\gamma_j-1) e_j} \right).
  \end{equation}
  Thus, $g(z) \in \Z_{\ge0}[z].$
  Dividing \eqref{equ:2-2k-1-g} by $(1 + z^{c_1})$ gives
  \begin{align}
    g(z)=& \ \ \ \ 1+ z^{2c_1-1} + z^{2(2c_1-1)} + \dots + z^{(k-2)(2c_1-1)} \label{equ:gz1} \\ 
         & + z^{c_1+c_2}\left( 1+ z^{2c_1-1} + z^{2(2c_1-1)} + \dots + z^{(k-2)(2c_1-1)}\right) \label{equ:gz2} \\
         & + \frac{z^{(k-1)(2c_1-1)} + z^{c_2+1}}{z^{c_1} + 1}.  \label{equ:gz3}
  \end{align}
  Since $z^{c_1}+1$ is a factor of $z^a + z^b$ if and only if $a-b$ is an odd multiple of $c_1,$ we have that 
  \[ c_2 + 1 = (k-1)(2c_1- 1) + (2m+1) c_1, \quad \text{for some integer $m$.}\]
  If $m=0,$ then we recover the situations given by Theorem \ref{thm:case1} with $a=2.$ Therefore, it is left to show that it is impossible to have $m \neq 0,$ which we prove by contradiction.

  Suppose $m > 0.$ Then the part \eqref{equ:gz3} of $g(z)$ becomes
  \[ z^{(k-1)(2c_1-1)}\left( 1-z^{c_1} + z^{2c_1} - \dots - z^{(2m-1)c_1} + z^{2mc_1} \right).\]
  As $m > 0,$ we see that the summand $- z^{(k-1)(2c_1-1) + c_1}$ with a negative coefficient appears in the above expression. 
Since $g(z)$ has non-negative coefficients, at least one summand in either \eqref{equ:gz1} or \eqref{equ:gz2} should have power $(k-1)(2c_1-1) + c_1.$ 
However, every exponent in \eqref{equ:gz1} is less than $(k-1)(2c_1-1)$ and every exponent appearing in \eqref{equ:gz2} is no less than $c_1 + c_2.$ 
However, we have
  \[ (k-1)(2c_1-1) < (k-1)(2c_1-1) + c_1 \le (c_2+1) - 3c_1 + c_1 < c_1 + c_2,\]
  a contradiction.

  Suppose $m < 0.$ For convenience, let $m' = -(m+1) \ge 0.$ Then $2m+1 = 2(m+1) -1 = -(2m'+1),$ and thus the part \eqref{equ:gz3} of $g(z)$ becomes
  \begin{equation}\label{equ:gz3'} z^{c_2+1}\left( 1-z^{c_1} + z^{2c_1} - \dots - z^{(2m'-1)c_1} + z^{2m'c_1} \right).
  \end{equation}
  We consider two cases.

  Suppose $c_2 + 1$ is not a multiple of $2c_1 -1.$ Note that this implies that $c_1 > 1.$ 
One can show, using Lemma \ref{lem:basic-geofac} part \eqref{itm:e1mu1}, that $e_2,$ the smallest exponent in the geometric factorization \eqref{equ:g(z)-geofac} of $g(z)$, is $\min(2c_1-1, c_2 + 1).$ 
Then it follows from Lemma \ref{lem:basic-geofac} part \eqref{itm:mtp_mu1} that $\max(2c_1-1,c_2+1) = e_j$ for some $j \ge 3.$ 
Thus, $z^{(2c_1-1)+(c_2+1)} = z^{2c_1+c_2}$ has to be a term appearing in $g(z).$ 
However, $z^{2c_1+c_2}$ is neither a term in \eqref{equ:gz1} since $c_2 + 1$ is not a multiple of $2c_1 -1,$ nor a term in \eqref{equ:gz3'} as $c_2 + 1 < 2c_1 + c_2 < 2c_1 + c_2 + 1.$ 
Hence, it must appear in \eqref{equ:gz2}. 
Thus, $2c_1 + c_2 = c_1 + c_2 + n(2c_1-1)$ for some non-negative integer $n.$ Then $c_1 = n(2c_1-1).$
Since $c_1 > 1,$ we deduce that $n=0$ and then $c_1 = 0,$ which is a contradiction.

  Suppose $c_2 +1$ is a multiple of $2c_1 -1.$ 
We first show that $c_1$ has to be $1.$ 
If $m' > 0,$ then the summand $-z^{c_1 + c_2 +1}$ with a negative coefficient appearing in \eqref{equ:gz3'}. 
Similarly to our prior argument, at least one summand in either \eqref{equ:gz1} or \eqref{equ:gz2} should have power $c_1 + c_2 +1.$ 
The only possible term in \eqref{equ:gz2} that could have the desired power is $z^{c_1 + c_2 + (2c_1-1)},$ which would imply $c_1 = 1.$ 
If a term in \eqref{equ:gz1} has the desired power, then we get that $c_1 + c_2 + 1$ is a multiple of $2c_1-1$ as well, which implies that $c_1 = (c_1+c_2+1)-(c_2+1)$ is a multiple of $2c_1 -1.$ 
It then follows that $c_1 = 1.$ 
Now we assume $m' = 0.$ 
Then $2m+1=-(2m'+1)=-1,$ and we have $c_2 + 1 = (k-1)(2c_1-1) -c_1.$ 
Thus, $c_1 + c_2 + 1 = (k-1)(2c_1-1)$ is a multiple of $2c_1-1$ again. 
Then similar to above, we have $c_1=1.$
Therefore, in all cases, we have shown that $c_1 = 1.$ 
 Plugging $c_1 =1$ into the expressions we have for $g(z),$ we can show (in two cases $m' > 0$ and $m' =0$) that
  \[ [z^{c_2}] g =1, \quad [z^{c_2+1}] g = 3, \quad [z^{c_2+2}] g = 1.\]  
  Noting that $1+1 < 3$ and observing that $1$ has to be an exponent in any factorization of $g(z),$ we find a contradiction to Lemma \ref{lem:basic-geofac} part \eqref{itm:sumcoeff}. 
This completes our proof.
\end{proof}



\section{Conjectures  and Questions}\label{sec:questions}

In this concluding section, we present a variety of conjectures and questions based on experimental evidence.


\subsection{Classifying Kronecker $h^*$-Polynomials When $\br=(a,ka-1)$}

In an exhaustive search of all $\bq$ supported on $\br=(r_1,r_2)$ with R-multiplicity $\bx=(x_1,x_2)$ where $1\leq r_i\leq 40$ and $1\leq x_i\leq 100$, the only $\bq=(\br,\bx)$ corresponding to Kronecker $h^*(\Pq;z)$ that are not covered by our results in Section~\ref{sec:twointegers} are given in Table~\ref{tab:exceptions}.
Based on these experiments, we offer the following conjecture and question.

\begin{conjecture}\label{conj:classify}
  For the family of $\bq$-vectors supported on two integers:
  \begin{enumerate}
  \item Section~\ref{sec:twointegers} describes all of the $\bq$-vectors supported on $\br$ of the form $(a,ka-1)$ or $(a-1,a)$ such that $h^*(\Pq;z)$ factors as a product of geometric series in powers of $z$, with the exception of the twelve $(\br,\bx)$-pairs of this form listed in Table~\ref{tab:exceptions}.
  \item For each vector $\br=(r_1,r_2)$ that is not of the form $(a,ka-1)$, there are only finitely many $\bx$ such that $\bq=(\br,\bx)$ has a Kronecker $h^*(\Pq;z)$.
  \end{enumerate}
\end{conjecture}

\begin{question}
  Is it true that when $\bq=(\br,\bx)$ is supported on two integers, $g_{\br}^{\bx}(z)$ is a geometric series in powers of $z$ if and only if $\br=(1,a)$ or $(a,1)$?
\end{question}

\begin{table}
  \centering
  \begin{tabular}{|l| l|} 
    \hline
    $\br$ & $\bx$ \\ [0.5ex] 
    \hline
    $(3, 7)$   & $(9, 14)$\\
    $(3, 10)$  & $(3, 5)$\\
    $(5, 7)$   & $(25, 7)$\\
    $(5, 8)$   & $(35, 13)$\\
    $(5,13)$   & $(5,13)$\\
    $(5, 17)$  & $(10, 17)$\\
    $(5, 18)$  & $(25, 18)$\\
    $(7, 9)$   & $(14, 3)$\\
    $(7, 11)$  & $(14, 33)$\\
    $(7, 33)$  & $(14, 11)$\\
    $(10, 17)$ & $(5, 17)$\\
    $(11, 14)$ & $(33, 7)$\\
    $(11, 26)$ & $(33, 52)$\\
    $(13, 18)$ & $(65, 18)$\\
    $(13, 34)$ & $(13, 34)$\\
    $(17, 29)$ & $(17, 58)$\\
    $(26, 33)$ & $(52, 11)$ \\ [1ex]
    \hline
  \end{tabular}
  \quad
  \begin{tabular}{|l| l|} 
    \hline
    $\br$ & $\bx$ \\ [0.5ex] 
    \hline
    $(2, 5)$ & $(7, 5)$ \\
    $(2, 7)$ & $(10, 7)$ \\
    $(2, 9)$ & $(4, 3)$ \\
    $(3, 4)$ & $(9, 11)$ \\
    $(3, 5)$ & $(13, 10)$ \\
    $(3, 8)$ & $(5, 4)$ \\
    $(3, 8)$ & $(21, 13)$\\
    $(3, 14)$& $(9, 7)$ \\
    $(4, 5)$ & $(6, 7)$ \\
    $(4, 5)$ & $(11, 15)$ \\
    $(5, 6)$ & $(7, 9)$ \\
    $(5, 9)$ & $(7, 6)$ \\
          &  \\
          & \\
          &  \\ [1ex]
    \hline
  \end{tabular}
  \caption{Pairs $\br$ and $\bx$ where $\bq=(\br,\bx)$ has Kronecker $h^*$-polynomial, but $\bq$ is not covered by a theorem in Section~\ref{sec:twointegers}. These are aggregated by whether or not $\br$ is of one of the forms $(a,ka-1)$ or $(a-1,a)$.}
  \label{tab:exceptions}
\end{table}


\subsection{Do Geometric Factorizations Classify Most Kronecker $h^*$-Polynomials?}

The $\bq$-vector given by $(\br,\bx)=((5,7), (25,7))$ has a Kronecker $h^*$-polynomial that does not factor into geometric series in powers of $z$, but it is the only known $\bq$-vector with this property.
Given Theorem~\ref{thm:2odd} and this experimental evidence, we make the following conjecture.

\begin{conjecture}\label{conj:kronfactor}
  For all but finitely many $\bq$-vectors supported on two integers, the polynomial $h^*(\Pq;z)$ is Kronecker if and only if it factors as a product of geometric series in powers of $z$.
\end{conjecture}

It seems feasible that the proof technique for Theorem~\ref{thm:2odd} might be extended to handle this general setting.
However, it has proven a challenge to find a universal way to handle all $\br$-vectors, either simultaneously or partitioned as a reasonable collection of sub-families.


\subsection{A Fibonacci Phenomenon}

The appearance of $((5,13),(5,13))$ and $((13,34),(13,34))$ in Table~\ref{tab:exceptions} suggests a more general phenomenon involving Fibonacci numbers.
Let $a_0=1$, $a_1=2$, and define $a_n=3a_{n-1}-a_{n-2}$.
Thus, the values $a_n$ correspond to ``every other'' Fibonacci number.
The following conjecture has been verified for $n\leq 7$.

\begin{conjecture}\label{conj:fib}
  Let $\bq$ be defined by $\br=\bx=(a_{n+1},a_{n})$.
  Then
\[
g_{(a_{n+1},a_{n})}^{(a_{n+1},a_{n})}(z)=\left( \sum_{i=0}^{a_n-1}z^i\right)\left( \sum_{i=0}^{a_{n+1}-1}z^i\right) \, .
\]
\end{conjecture}

There are several unique aspects of Conjecture~\ref{conj:fib} that distinguish it from the theorems where $\br=(a,ka-1)$.
First, in the factorizations found in the $\br=(a,ka-1)$ setting, the $\br$-vector was fixed and the $\bx$-vector was varying.
For this conjecture, both $\br$ and $\bx$ are varying simultaneously.
Second, the arithmetical structure of the $\br$- and $\bx$-vectors in the $(a,ka-1)$ setting are considerably simpler than in this context.
For example, consider the following lemma.

\begin{lemma}\label{lem:fibprop}
  The following properties hold for the sequence $(a_n)$.
  \begin{enumerate}
  \item For $n\geq 2$, $1+a_{n-1}^2=a_na_{n-2}$.
  \item For $n\geq 0$, $1+a_{n}^2+a_{n+1}^2=3a_na_{n+1}$, and thus $\bx=(a_n,a_{n+1})$ is an R-multiplicity for $\br=(a_n,a_{n+1})$ with $\ell=3$ and the corrsponding $\Pq$ is reflexive.
  \item $\gcd(a_n,a_{n+1})=1$.
  \item{ For $\br=\bx=(a_{n+1},a_n)$ and $\bi=(i_1,i_2)\in \res{a_n}\times \res{a_{n+1}}$, we have
      \begin{align*}
        u(\alpha(\bi))& =3i_1+a_{n-1}w_1(\bi)-a_nw_2(\bi) \\
                      & = 3i_1+a_{n-1}(a_n(i_1-i_2)\bmod a_{n+1})-a_n(a_{n-1}(i_1-i_2)\bmod a_n)\, .
      \end{align*}}
  \end{enumerate}
\end{lemma}

\begin{proof}
  \commentout{
    For the first item, observe that $1+2^2=1\cdot 5$ and $1+5^2=2\cdot 13$.
    By induction using these as base cases, we have
    \begin{align*}
      a_n\cdot a_{n-2} & = (3a_{n-1}-a_{n-2})a_{n-2} \\
                       & = 3a_{n-1}a_{n-2}-a_{n-2}^2 \\
                       & = 3a_{n-1}a_{n-2}-(a_{n-1}a_{n-3}-1)\\
                       & = a_{n-1}(3a_{n-2}-a_{n-3})+1 \\
                       & = a_{n-1}^2+1 \, .
    \end{align*}

    The equality $1+a_{n}^2+a_{n+1}^2=3a_na_{n+1}$ is obtained using the identities $1+a_{n}^2=a_{n-1}a_{n+1}$ and $a_{n+1}+a_{n-1}=3a_n$ as follows:
    \begin{align*}
      1+a_{n}^2+a_{n+1}^2 & = a_{n+1}^2+a_{n+1}a_{n-1} \\
                          & = a_{n+1}(a_{n+1}+a_{n-1}) \\
                          & = 3a_na_{n+1}
    \end{align*}
    Finally, since $1=a_na_{n-2}-a_{n-1}^2$, it follows that $\gcd(a_n,a_{n+1})=1$.
  }

  The first three claims follow from straightforward arguments using induction and application of the defining identity for $a_n$.
  For the fourth item, since $-1 = a_na_n-a_{n-1}a_{n+1},$ we have that $\rho_1=-a_{n-1}$ and $\rho_2=a_n$.
  Thus, since $a_{n+1}=3a_n-a_{n-1}$, we have that $c_1=3$, and since $a_n<a_{n+1}$ we have $c_2=0$.
  The result follows from Theorem~\ref{thm:gxrthm}.
\end{proof}

The fact that $\ell=3$ for all $n$ establishes that $(1+z+z^2)$ is a factor of the $h^*$-polynomial in this case, and thus one expects that $g_{(a_{n+1},a_n)}^{(a_{n+1},a_n)}(z)$ factors as a product of two geometric series.
However, the behavior of $u(\alpha(\bi))$ is quite subtle, in the following sense.
For $\bi=(i_1,i_2)\in \res{a_n}\times \res{a_{n+1}}$, define
\[
  v(\bi):=a_{n-1}(a_n(i_1-i_2)\bmod a_{n+1})-a_n(a_{n-1}(i_1-i_2)\bmod a_n) \, ,
\]
so that $u(\alpha(\bi))=3i_1+v(\bi)$.
Thus, for all $(i_1,i_2)$, we have
\[ 
  v(i_1,i_2)=v(i_1+1,i_2+1) \, ,
\]
and hence
\[
  u(\alpha(i_1+1,i_2+1))=3+u(\alpha(i_1,i_2)) \, .
\]
This implies that the values of $u(\alpha(\bi))$ are essentially determined by the boundary values $u(\alpha(i_1,0))$ and $u(\alpha(0,i_2))$.
Experimental data combined with an OEIS~\cite{OEIS} search leads us to the following conjecture. 

\begin{conjecture}\label{conj:fibu}
  \begin{enumerate}
  \item The value of $u(\alpha(\bi))$ is independent of $n$.
  \item For all $i_1\geq 0$, we have
    \[
      u(\alpha(i_1,0))=\left\lceil i_1\left(\frac{1+\sqrt{5}}{2}\right)^2  \right\rceil \, .
    \]
  \item For all $i_2\geq 0$, we have
    \[
      u(\alpha(0,i_2))=2i_2-\left\lfloor i_2\left( \frac{1+\sqrt{5}}{2}\right) \right\rfloor \, .
    \]
  \end{enumerate}
\end{conjecture}

Some values of $u(\bi):=u(\alpha(\bi))$ are given in Table~\ref{tab:fibu}.
It seems that obtaining a more precise understanding of Conjecture~\ref{conj:fibu} and the values of $u(\bi)$ is needed to resolve Conjecture~\ref{conj:fib}.

\begin{table}
  \centering
  \begin{tabular}{|l|l|l|l|l|l|l|l|l|l|l|l|l|}
    \hline
    $0$ & $1$ & $1$ & $2$ & $2$ & $2$ & $3$ & $3$ & $4$ & $4$ & $4$ & $5$ & $5$ \\
    \hline
    $3$ & $3$ & $4$ & $4$ & $5$ & $5$ & $5$ & $6$ & $6$ & $7$ & $7$ & $7$ & $8$ \\
    \hline
    $6$ & $6$ & $6$ & $7$ & $7$ & $8$ & $8$ & $8$ & $9$ & $9$ & $10$ & $10$ & $10$ \\
    \hline
    $8$ & $9$ & $9$ & $9$ & $10$ & $10$ & $11$ & $11$ & $11$ & $12$ & $12$ & $13$ & $13$ \\
    \hline
    $11$ & $11$ & $12$ & $12$ & $12$ & $13$ & $13$ & $14$ & $14$ & $14$ & $15$ & $15$ & $16$ \\
    \hline
  \end{tabular}
  \caption{Some values of $u(\alpha(i_1,i_2))$ with $i_1\geq 0$ indexing rows and $i_2\geq 0$ indexing columns.}
  \label{tab:fibu}
\end{table}


\subsection{On Ehrhart Positivity}

We conjecture that independent of the reflexivity condition, all $\Pq$ with $\bq$ supported by two integers are Ehrhart positive.

\begin{conjecture}\label{conj:twopositive}
  All $\Pq$ with $\bq$ supported on two integers are Ehrhart positive.
\end{conjecture}

Conjecture~\ref{conj:twopositive} has been verified for all $\bq=(\br,\bx)$ with $1\leq r_i\leq 15$ and $1\leq x_i\leq 24$.
Note that this general Ehrhart positivity is not a result of only Theorem~\ref{thm:geomriemhyp} and Kronecker polynomial techniques, as most $\Pq$ are not reflexive.

\subsection{$\bq$-Vectors Supported on Three Integers}

A natural next step is to consider $\bq$ that are supported by more than two integers.
Experimental computation and Proposition~\ref{prop:formula_for_alpha_omega} suggest that a starting point for such an exploration are $3$-supported $\bq$'s with $\bs$ entries coprime.
When $\gcd(a,b)=1$ and $\bs=(b,a,1)$, so that $\br=(a,b,ab)$, Theorem~\ref{thm:lcmextend} implies that this reduces to the case where $\br=(a,b)$.
Thus, we can consider only those $\bs$ such that the $s_i$ are pairwise coprime and each $s_i\geq 2$.
The first such example is $\bs=(5,3,2)$, for which we have the following result.

\begin{theorem}\label{thm:532}
  Let $\bs=(5,3,2)$, $\br=(6,10,15)$, and $\bx=(5c_1-1,3c_2-1,2c_3+1)$ for $c_1,c_2\geq 1$ and $c_3\geq 0$.
  For $\bq=(\br,\bx)$, the following three cases imply that $h^*(\Pq;z)$ is Kronecker.
  \begin{enumerate}
  \item $(c_1,c_2,c_3)=(1,3,1)$, where 
    \[
      g_{(6,10,15)}^{(4,8,3)}(z)=(1+z^3)(1+z^2+z^4)(1+z+z^2+z^3+z^4)^2
    \]
  \item $(c_1,c_2,c_3)=(c,c,4c-1)$ for $c\geq 1$, where 
    \[
      g_{(6,10,15)}^{(5c-1,3c-1,2(4c-1)+1)}(z)=(1+z^{4c-1})(1+z^c+z^{2c})(1+z+z^c+z^{2c}+z^{3c}+z^{4c})
    \]
  \item $(c_1,c_2,c_3)=(c,3c,7c-1)$ for $c\geq 1$, where
    \[
      g_{(6,10,15)}^{(5c-1,3(3c)-1,2(7c-1)+1)}(z)=(1+z^{7c-1})(1+z^{3c}+z^{6c})(1+z+z^c+z^{2c}+z^{3c}+z^{4c})
    \]
  \end{enumerate}
\end{theorem}

\begin{proof}
  We sketch the proof.
  The case (1) is straightforward to verify directly.
  For case (2), we use a similar technique to those used for the proofs in Section~\ref{sec:twointegers} where we identify a bijection of $\res{5}\times\res{3}\times\res{2}$ that yields the factorization.
  In this case, if we fix all elements except for the following pairs which are exchanged by the bijection, then the factorization follows:
  \begin{align*}
    & (2,2,0)\longleftrightarrow (0,0,1), \, \, \, (4,1,0)\longleftrightarrow (1,0,1)\\
    & (3,2,0)\longleftrightarrow (0,1,1), \, \, \, (4,2,0)\longleftrightarrow (2,0,1)
  \end{align*}
  Similarly for case (3), if we fix all elements except for the following pairs which are exchanged by the bijection, then the factorization follows:
  \begin{align*}
    & (2,2,0)\longleftrightarrow (1,0,1), \, \, \,  (4,1,0)\longleftrightarrow (0,0,1)\\
    & (3,2,0)\longleftrightarrow (2,0,1), \, \, \, (4,2,0)\longleftrightarrow (0,1,1) \qedhere
  \end{align*}
\end{proof}

Experimental evidence suggests that these are the only $\bq$ supported on $(6,10,15)$ with Kronecker $h^*$-polynomials.
A search over (pairwise coprime) $\bs$ and $\bx$ with $2\leq s_i\leq 11$ and $1\leq x_i\leq 50$ has produced only two further examples of $3$-supported $\bq$'s with Kronecker $h^*$-polynomials, specifically:
\[
  \bs=(11,4,3), \, \br=(12, 33, 44), \, \bx=(21, 11, 22)
\]
and
\[
  \bs=(10,7,3), \, \br=(21, 30, 70), \, \bx=(9, 10, 5)
\]
For both $\bs=(11,4,3)$ and $\bs=(10,7,3)$, there are no other associated $\Pq$ with Kronecker $h^*$-polynomials for any $\bx$ with each $1\leq x_i\leq 75$.
Hence, we present the following question.

\begin{question}
  Are there other general families of $\bq$-vectors supported on more than two integers such that $h^*(\Pq;z)$ is Kronecker?
  In particular, are there other $3$-supported $\bq$'s with $\bs$ entries coprime that have Kronecker $h^*(\Pq;z)$?
\end{question}

\subsection{Properties of Factorizations}

Our main approach in this paper has been to study factorizations of $g_{\br}^{\bx}(z)$ into geometric series in powers of $z$.
However, as Remark~\ref{remark:h*} shows, it is possible for $h^*(\Pq;z)$ to have a geometric factorization for $\bq=(\br,\bx)$, yet for $g_{\br}^{\bx}(z)$ to not have such a factorization, leading to the following question.
\begin{question}
  Are there only finitely many $\bq=(\br,\bx)$ with Kronecker $h^*(\Pq;z)$ where $h^*(\Pq;z)$ admits a geometric factorization, but $g_{\br}^{\bx}(z)$ does not?
\end{question}

If a polynomial is Kronecker, then it factors into cyclotomic factors.
It would be interesting to determine how these factors are related to $\bq$ in the case of $h^*$-polynomials, hence the following question.
\begin{question}
  How, if at all, is the factorization of a Kronecker $h^*(\Pq;z)$ into cyclotomic factors related to arithmetic properties of $\bq$?
\end{question}

\bibliographystyle{plain}
\bibliography{Braun}

\end{document}